\documentclass[global]{svjour}%
\usepackage{amsfonts}
\usepackage{amsmath}
\usepackage{amssymb}
\usepackage{graphicx}%
\providecommand{\U}[1]{\protect\rule{.1in}{.1in}}
\begin{document}

\journalname{ArXIV}
\title{A Criterion for the Viability of Stochastic Semilinear Control Systems via the
Quasi-Tangency Condition }
\author{Dan Goreac}
\institute{Universit\'{e} Paris-Est, Laboratoire d'Analyse et Math\'{e}matiques
Appliqu\'{e}es, UMR 8050, 5, Boulevard Descartes, Cit\'{e} Descartes -
Champs-sur-Marne, 77454 Marne-la-Vall\'{e}e, Cedex 2 \
\email{Dan.Goreac@univ-mlv.fr}%
}
\maketitle

\begin{abstract}
In this paper we study a criterion for the viability of stochastic semilinear
control systems on a real, separable Hilbert space. The necessary and
sufficient conditions are given using the notion of stochastic quasi-tangency.
As a consequence, we prove that approximate viability and the viability
property coincide for stochastic linear control systems. We obtain Nagumo's
stochastic theorem and we present a method allowing to provide explicit
criteria for the viability of smooth sets. We analyze the conditions
characterizing the viability of the unit ball. The paper generalizes recent
results from the deterministic framework.

\end{abstract}



\section{Introduction}

We begin by introducing the basic notations. We consider two separable real
Hilbert spaces $\left(  H,\left\langle \cdot,\cdot\right\rangle _{H}\right)
,$ $\left(  \Xi,\left\langle \cdot,\cdot\right\rangle _{\Xi}\right)  $. We let
$\mathcal{L}\left(  \Xi;H\right)  $ be the space of continuous linear
operators and $L_{2}\left(  \Xi;H\right)  $ be the space of Hilbert-Schmidt
linear operators endowed with its usual norm. We consider a linear operator
$A:D\left(  A\right)  \subset H\longrightarrow H$ which generates a $C_{0}%
$-semigroup of linear operators $\left(  S\left(  t\right)  \right)  _{t\geq
0}.$ We let $\left(  \Omega,\mathcal{F},P\right)  $ be a complete probability
space. The process $W$ will denote a cylindrical Wiener process with values in
$\Xi$. The probability space $\left(  \Omega,\mathcal{F},P\right)  $ is
endowed with the natural, complete filtration $\left(  \mathcal{F}_{t}\right)
_{t\geq0}$ generated by $W$. We consider $\left(  G,\left\langle \cdot
,\cdot\right\rangle _{G}\right)  $ a real separable Hilbert space and a
closed, bounded subset $U\subset G.$ For a finite time horizon $T>0$, we let
$\mathcal{A}$ denote the space of all predictable processes $u:\left[
0,T\right]  \times\Omega\longrightarrow U.$ The coefficient functions
$f:H\times U\longrightarrow H$ and $g:H\times U\longrightarrow\mathcal{L}%
\left(  \Xi;H\right)  $ will be supposed to satisfy standard assumptions (see
Section 2). Finally, we consider a closed set $K\subset H$.

Given a stochastic control system%
\begin{equation}
\left\{
\begin{array}
[c]{l}%
dX^{t,\xi,u}(s)=\left(  AX^{t,\xi,u}(s)+f\left(  X^{t,\xi,u}(s),u\left(
s\right)  \right)  \right)  ds\\
+g\left(  X^{t,\xi,u}(s),u\left(  s\right)  \right)  dW_{s},\text{ for all
}s\in\left[  t,T\right]  ,\\
X^{t,\xi,u}(t)=\xi\in L^{2}\left(  \Omega,\mathcal{F}_{t},P;H\right)  ,
\end{array}
\right.  \label{SDE1}%
\end{equation}
the aim of this paper is to give necessary and sufficient conditions for
which, for every $t\in\left[  0,T\right]  $, and every $\xi\in L^{2}\left(
\Omega,\mathcal{F}_{t},P;K\right)  $, one can find an admissible control
process $u\in\mathcal{A}$ such that the mild solution of (\ref{SDE1})
associated to $u$ remains inside the set $K,$ or, at least, in an arbitrarily
small neighborhood of $K$. These properties are called viability, respectively
$\varepsilon-$viability, and they have been extensively studied both in
deterministic and stochastic setting. In the finite-dimensional deterministic
framework, the first result on viability goes back to Nagumo \cite{N} and it
has been rediscovered several times in the late sixties. For stochastic
finite-dimensional systems, the methods used to characterize viability rely
either on stochastic contingent cones (e.g. \cite{A}-\cite{AF}, \cite{GT}) or
on viscosity solutions (e.g. \cite{BCQ}-\cite{BQR}, \cite{G}). \ We also
recall \cite{BQT} for a necessary condition for the viability of semilinear
evolution systems using viscosity solutions of a class of fully nonlinear
Hamilton-Jacobi-Bellman equations in abstract Hilbert spaces.

In \cite{DG2} it has been proved by duality methods that approximate
controllability of infinite-dimensional linear systems reduces to the study of
(backward) viability with respect to some dual system. Therefore, beside the
interest in the result itself, a necessary and sufficient criterion for
viability in the infinite-dimensional setting could be a tool for the study of
controllability properties.

Recently, the authors of \cite{CV}-\cite{CNV} have provided a characterization
of the viability of (deterministic) multi-valued nonlinear evolutions on
Banach spaces via the quasi-tangency condition. Motivated by this approach,
our main objective is to introduce the notion of quasi-tangency corresponding
to the stochastic framework and to prove the relation with the viability of
stochastic semilinear control systems. This will allow to extend existing
results (e.g. \cite{ADP}) in both directions: to an infinite-dimensional state
space and to systems driven by cylindrical Brownian motion. Under the
assumption that the noise coefficient $g$ is $L_{2}\left(  \Xi;H\right)
$-valued, we provide a criterion equivalent to the $\varepsilon$-viability
property. In the more general framework, this condition remains sufficient for
$\varepsilon$-viability. For the converse, the estimates only allow to prove a
slightly weaker quasi-tangency condition.

We point out that the stochastic quasi-tangency conditions extend the
stochastic contingent cone introduced in \cite{ADP} for the finite-dimensional
setting. As for \cite{ADP}, the criteria are not deterministic (as is the case
in \cite{BCQ}-\cite{BQR}, \cite{G}), but the method can be easily adapted for
random sets of constraints. One can derive deterministic conditions by
calculating contingent sets to direct images. We will give a simple example
showing how the viability of the unit ball can be explicitly characterized
from the quasi-tangency conditions.

The paper is organized as follows: In the first section, we introduce the
concept of quasi-tangency and state the main results. The second section is
concerned with the proof of the equivalence between stochastic quasi-tangency
and the property of $\varepsilon$-viability for stochastic semilinear control
systems. In the last section, two classes of examples are considered. First,
we prove that, for infinite-dimensional stochastic linear control systems,
$\varepsilon$-viability and viability coincide. Second, we give Nagumo's
stochastic theorem as a corollary of our main result and deduce explicit
conditions for the viability of smooth sets. In particular, we study the
viability of the unit ball in $H$.

\section{Assumptions and main results}

The coefficient functions $f:H\times U\longrightarrow H$ and $g:H\times
U\longrightarrow\mathcal{L}\left(  \Xi;H\right)  $ are supposed to satisfy the
following standard assumptions:

\textbf{(A1)} There exists some positive constant $c>0$ such that%
\begin{equation}
\left\{
\begin{array}
[c]{l}%
\left\vert f\left(  x,u\right)  -f\left(  y,u\right)  \right\vert \leq
c\left\vert x-y\right\vert ,\text{ and }\\
\left\vert f\left(  x,u\right)  \right\vert \leq c\left(  1+\left\vert
x\right\vert \right)  ,
\end{array}
\right.  \tag{A1}\label{A1}%
\end{equation}
for all $x,y\in H$ and all $u\in U;$

\textbf{(A2)} For every $t>0$ and every $\left(  x,u\right)  \in H\times U$,
one has $S\left(  t\right)  g\left(  x,u\right)  \in L_{2}\left(
\Xi;H\right)  .$ Moreover, for some constants $c>0$ and $0\leq\gamma<\frac
{1}{2},$
\begin{equation}
\left\{
\begin{array}
[c]{l}%
\left\vert g\left(  x,u\right)  \right\vert _{\mathcal{L}\left(  \Xi;H\right)
}\leq c\left(  1+\left\vert x\right\vert \right)  ,\\
\left\vert S\left(  t\right)  \left(  g\left(  x,u\right)  -g\left(
y,u\right)  \right)  \right\vert _{L_{2}\left(  \Xi;H\right)  }\leq c\left(
t\wedge1\right)  ^{-\gamma}\left\vert x-y\right\vert ,\text{ and}\\
\left\vert S\left(  t\right)  g\left(  x,u\right)  \right\vert _{L_{2}\left(
\Xi;H\right)  }\leq c\left(  t\wedge1\right)  ^{-\gamma}\left(  1+\left\vert
x\right\vert \right)  ,
\end{array}
\right.  \text{ } \tag{A2}\label{A2}%
\end{equation}
for all $x,y\in H$ and all $u\in U.$

Given $t\in\left[  0,T\right]  $ and an admissible control process
$u\in\mathcal{A}$, we recall that an $\left(  \mathcal{F}_{t}\right)
$-predictable process $X^{t,\xi,u}$ with $E\left[  \sup_{s\in\left[
t,T\right]  }\left\vert X^{t,\xi,u}\left(  s\right)  \right\vert ^{2}\right]
<\infty$ is a mild solution of (\ref{SDE1}) if, for all $s\in\left[
t,T\right]  ,$%
\begin{align*}
X^{t,\xi,u}\left(  s\right)   &  =S\left(  s-t\right)  \xi+\int_{t}%
^{s}S\left(  s-r\right)  f\left(  X^{t,\xi,u}\left(  r\right)  ,u\left(
r\right)  \right)  dr\\
&  +\int_{t}^{s}S\left(  s-r\right)  g\left(  X^{t,\xi,u}\left(  r\right)
,u\left(  r\right)  \right)  dW_{r},\text{ }dP-a.s.
\end{align*}
Under the standard assumptions (\ref{A1}) and (\ref{A2}), there exists a
unique mild solution of (\ref{SDE1}). For further results on mild solutions,
the reader is referred to \cite{DPZ} and \cite{FT}.

Let us consider $\lambda\in\left[  0,\frac{1}{2}\right)  $ and introduce the
concept of stochastic $\lambda$ quasi-tangency.

\begin{definition}
($\lambda$ Quasi-tangency condition) A closed set $K\subset H$ satisfies the
$\lambda$ quasi-tangency condition with respect to the control system
(\ref{SDE1}) if, for every $t\in\left[  0,T\right)  $ and every $\xi\in
L^{2}\left(  \Omega,\mathcal{F}_{t},P;K\right)  ,$ we have%
\begin{align}
\underset{h\searrow0}{\lim\inf}\inf &  \left\{  \frac{1}{h^{1-2\lambda}%
}E\left[  \left\vert \zeta-\eta\right\vert ^{2}\right]  +\frac{1}{h^{2}%
}E\left[  \left\vert E^{\mathcal{F}_{t}}\left[  \zeta-\eta\right]  \right\vert
^{2}\right]  :\right. \nonumber\\
&  \left.  \left(  \zeta,\eta\right)  \in\mathcal{S}\left(  t,h\right)
\xi\text{ }\times\text{ }L^{2}\left(  \Omega,\mathcal{F}_{t+h},P;K\right)
\right\}  =0, \label{QTC1}%
\end{align}
where%
\begin{align*}
\mathcal{S}\left(  t,h\right)  \xi=  &  \left\{  S\left(  h\right)  \xi
+\int_{t}^{t+h}S\left(  t+h-s\right)  f\left(  \xi,u(s)\right)  ds,\right. \\
&  \left.  +\int_{t}^{t+h}S\left(  t+h-s\right)  g\left(  \xi,u(s)\right)
dW_{s},\text{ }u\in\mathcal{A}\right\}  .
\end{align*}

\end{definition}

\begin{remark}
The term involving the conditional expectation in (\ref{QTC1}) corresponds to
the deterministic quasi-tangency condition. The term in $h^{1-2\lambda}$
(written for $\lambda=\gamma)$ is specific to the stochastic part of the
equation (\ref{SDE1}). Whenever the coefficient function $g$ takes its values
in $L_{2}\left(  \Xi;H\right)  $ and is Lipschitz continuous in the state
variable, one can consider $\lambda=0.$
\end{remark}

The following simple proposition provides a sequential formulation of the
stochastic quasi-tangency condition.

\begin{proposition}
A nonempty, closed set $K\subset H$ satisfies the $\lambda$ quasi-tangency
condition with respect to the control system (\ref{SDE1}) if and only if, for
every $t\in\left[  0,T\right)  $ and every $\xi\in L^{2}\left(  \Omega
,\mathcal{F}_{t},P;K\right)  $, there exist a sequence of positive real
constants $h_{n}\searrow0,$ a sequence of random variables $p_{n}\in
L^{2}\left(  \Omega,\mathcal{F}_{t+h_{n}},P;H\right)  $ and a sequence of
admissible control processes $\left(  u_{n}\right)  _{n}\subset\mathcal{A}$
such that the following assertions hold simultaneously:

(a) $\lim_{n}\left(  E\left[  \left\vert p_{n}\right\vert ^{2}\right]
+\frac{1}{h_{n}^{1+2\gamma}}E\left[  \left\vert E^{\mathcal{F}_{t}}\left[
p_{n}\right]  \right\vert ^{2}\right]  \right)  =0,$ and

(b) $S\left(  h_{n}\right)  \xi+\int_{t}^{t+h_{n}}S\left(  t+h_{n}-s\right)
f\left(  \xi,u_{n}(s)\right)  ds+$

$\int_{t}^{t+h_{n}}S\left(  t+h_{n}-s\right)  g\left(  \xi,u_{n}(s)\right)
dW_{s}+h_{n}^{\frac{1}{2}-\lambda}p_{n}\in K,$ $dP-$almost surely, for all
$n\geq1.$
\end{proposition}

The proof is straightforward and we leave the details to the interested
reader. We recall the definitions of $\varepsilon-$viability and viability.

\begin{definition}
(a) A nonempty, closed set $K\subset H$ is called (mild) viable with respect
to the control system (\ref{SDE1}) if, for every $t\in\left[  0,T\right]  $
and every initial condition $\xi\in L^{2}\left(  \Omega,\mathcal{F}%
_{t},P;K\right)  ,$ there exists an admissible control process $u$ such that
\[
X^{t,\xi,u}\left(  s\right)  \in K,\text{ }dP\text{-a.s., for all }s\in\left[
t,T\right]  .
\]

(b) A nonempty, closed set $K\subset H$ is called (mild) $\varepsilon-$viable
with respect to the control system (\ref{SDE1}) if, for every $t\in\left[
0,T\right]  $ and every initial condition $\xi\in L^{2}\left(  \Omega
,\mathcal{F}_{t},P;K\right)  ,$
\[
\inf_{u\in\mathcal{A}}\sup_{s\in\left[  t,T\right]  }d\left(  X^{t,\xi
,u}\left(  s\right)  ,L^{2}\left(  \Omega,\mathcal{F}_{s},P;K\right)  \right)
=0.
\]

\end{definition}

\begin{remark}
In the definition of $\varepsilon-$viability, the set $L^{2}\left(
\Omega,\mathcal{F}_{s},P;K\right)  $ should be seen as a closed subset of
$L^{2}\left(  \Omega,\mathcal{F}_{s},P;H\right)  $ and $d\left(  \cdot
,L^{2}\left(  \Omega,\mathcal{F}_{s},P;K\right)  \right)  $ as the usual
distance to a closed set in the Hilbert space $L^{2}\left(  \Omega
,\mathcal{F}_{s},P;H\right)  ,$ i.e.
\[
d^{2}\left(  \zeta,L^{2}\left(  \Omega,\mathcal{F}_{s},P;K\right)  \right)
=\inf\left\{  E\left[  \left\vert \zeta-\eta\right\vert ^{2}\right]  :\eta\in
L^{2}\left(  \Omega,\mathcal{F}_{s},P;K\right)  \right\}  ,
\]
for every $\zeta\in L^{2}\left(  \Omega,\mathcal{F}_{s},P;H\right)  .$
Whenever the projection map $\Pi_{K}:H\rightsquigarrow K$ defined by
\[
\Pi_{K}(x)=\left\{  y\in K:d_{K}(x)=\left\vert y-x\right\vert \right\}
\]
has nonempty images (e.g. for closed, convex sets of constraints), using
\cite{AF}, Corollary 8.2.13, one can replace the $\varepsilon$-viability
condition by
\[
\inf_{u\in\mathcal{A}}\sup_{s\in\left[  t,T\right]  }E\left[  d_{K}^{2}\left(
X^{t,\xi,u}\left(  s\right)  \right)  \right]  =0.
\]

\end{remark}

We now state the main results of the paper. The proofs will be postponed to
the next section.

\begin{theorem}
(Necessary condition for $\varepsilon$-viability)

Let us suppose that (\ref{A1}) and (\ref{A2}) hold true. Moreover, we suppose
that $K\subset H$ is a nonempty, closed set. If $K$ is $\varepsilon-$viable
with respect to the control system (\ref{SDE1}), then the quasi-tangency
condition (\ref{QTC1}) holds true for $\lambda=\gamma$.
\end{theorem}

\begin{theorem}
(Sufficient condition for $\varepsilon$-viability)

Let us suppose that (\ref{A1}) and (\ref{A2}) hold true. Moreover, we suppose
that $K\subset H$ is a nonempty, closed set and the quasi-tangency condition
(\ref{QTC1}) holds true for $\lambda=0$. Then the set $K$ is $\varepsilon
-$viable with respect to the control system (\ref{SDE1}).
\end{theorem}

The following result is a direct consequence of the previous theorems. It
provides a useful criterion equivalent to the $\varepsilon$-viability of
control systems with $L_{2}\left(  \Xi;H\right)  $-valued noise coefficient.

\begin{corollary}
We assume (\ref{A1}) and

\textbf{(A2') }For every $\left(  x,u\right)  \in H\times U$, one has
$g\left(  x,u\right)  \in L_{2}\left(  \Xi;H\right)  .$ Moreover, for some
real constant $c>0,$
\begin{equation}
\left\{
\begin{array}
[c]{l}%
\left\vert g\left(  x,u\right)  -g\left(  y,u\right)  \right\vert
_{L_{2}\left(  \Xi;H\right)  }\leq c\left\vert x-y\right\vert ,\text{ and}\\
\left\vert g\left(  x,u\right)  \right\vert _{L_{2}\left(  \Xi;H\right)  }\leq
c\left(  1+\left\vert x\right\vert \right)  ,
\end{array}
\right.  \text{ } \tag{A2'}\label{A2'}%
\end{equation}
for all $x,y\in H$ and all $u\in U.$

Then the quasi-tangency condition (\ref{QTC1}) with $\lambda=0$ provides a
criterion equivalent to the $\varepsilon$-viability property.
\end{corollary}

\section{Proof of the main results}

\subsection{Necessary condition for $\varepsilon$-viability}

\begin{proof}
(of Theorem 1). We begin by proving that, whenever $K$ enjoys the
$\varepsilon$-viability property, it satisfies the quasi-tangency condition
with $\lambda=\gamma$. We consider arbitrary $t\in\left[  0,T\right)  ,$
$h\in\left(  0,1\right)  $ small enough and $\xi\in L^{2}\left(
\Omega,\mathcal{F}_{t},P;K\right)  $. If $K$ enjoys the $\varepsilon
$-viability property, then there exists an admissible control process $u$
(which may depend on $t,$ $h$ and $\xi)$ such that the mild solution of
(\ref{SDE1}) issued from $\xi$ and associated to $u$ (denoted by $X^{t,\xi,u}%
$) satisfies%
\[
d\left(  X^{t,\xi,u}\left(  s\right)  ,L^{2}\left(  \Omega,\mathcal{F}%
_{s},P;K\right)  \right)  <h^{3}\text{ for all }s\in\left[  t,T\right]  .
\]
In particular, there exists a random variable $\eta\in L^{2}\left(
\Omega,\mathcal{F}_{t+h},P;K\right)  $ such that
\begin{equation}
E\left[  \left\vert X^{t,\xi,u}\left(  t+h\right)  -\eta\right\vert
^{2}\right]  <h^{3}. \label{In0}%
\end{equation}
Using the continuity property of the mild solution (see, for example
\cite{DPZ} proof of Theorem 9.9.1), we get%
\begin{equation}
\sup_{s\in\left[  t,t+h\right]  }E\left[  \left\vert X^{t,\xi,u}\left(
s\right)  -\xi\right\vert ^{2}\right]  \leq C\left(  \sup_{r\in\left[
0,h\right]  }E\left[  \left\vert S\left(  r\right)  \xi-\xi\right\vert
^{2}\right]  +h^{1-2\gamma}\right)  , \label{In1}%
\end{equation}
where $C$ is a generic constant which may change from one line to another. $C$
depends on the Lipschitz constants of the coefficients, the initial data $\xi$
and the time horizon $T>0$ (but not on $h$). Using the assumptions (A1) and
(A2) and the inequality (\ref{In1}), one gets%
\begin{align}
&  \sup_{s\in\left[  t,t+h\right]  }E\left[  \left\vert f\left(  X^{t,\xi
,u}\left(  s\right)  ,u\left(  s\right)  \right)  -f\left(  \xi,u(s)\right)
\right\vert ^{2}\right] \nonumber\\
&  \leq C\left(  \sup_{r\in\left[  0,h\right]  }E\left[  \left\vert S\left(
r\right)  \xi-\xi\right\vert ^{2}\right]  +h^{1-2\gamma}\right)  ,
\label{In10}%
\end{align}
and%
\begin{align}
&  E\left[  \left\vert S\left(  t+h-s\right)  \left(  g\left(  X^{t,\xi
,u}\left(  s\right)  ,u\left(  s\right)  \right)  -g\left(  \xi,u(s)\right)
\right)  \right\vert ^{2}\right] \nonumber\\
&  \leq C\left(  t+h-s\right)  ^{-2\gamma}\left(  \sup_{r\in\left[
0,h\right]  }E\left[  \left\vert S\left(  r\right)  \xi-\xi\right\vert
^{2}\right]  +h^{1-2\gamma}\right)  , \label{In10'}%
\end{align}
for all $s\in\left[  t,t+h\right]  $. Let us now introduce
\begin{align}
q_{h}=  &  \eta-S\left(  h\right)  \xi-\int_{t}^{t+h}S\left(  t+h-s\right)
f\left(  \xi,u(s)\right)  ds\nonumber\\
&  -\int_{t}^{t+h}S\left(  t+h-s\right)  g\left(  \xi,u(s)\right)  dW_{s}.
\label{qh}%
\end{align}
Combining (\ref{In10}), (\ref{In10'}) and (\ref{In0}) yields
\begin{align}
&  E\left[  \left\vert q_{h}\right\vert ^{2}\right] \nonumber\\
&  \leq C\left(  E\left[  \left\vert X^{t,\xi,u}\left(  t+h\right)
-\eta\right\vert ^{2}\right]  \right. \nonumber\\
&  \text{ \ \ }+E\left[  \left\vert \int_{t}^{t+h}S\left(  t+h-s\right)
\left[  f\left(  X^{t,\xi,u}(s),u(s)\right)  -f\left(  \xi,u(s)\right)
\right]  ds\right\vert ^{2}\right] \nonumber\\
&  \text{ \ \ }\left.  +E\left[  \left\vert \int_{t}^{t+h}S\left(
t+h-s\right)  \left[  g\left(  X^{t,\xi,u}(s),u(s)\right)  -g\left(
\xi,u(s)\right)  \right]  dW_{s}\right\vert ^{2}\right]  \right) \nonumber\\
&  \leq Ch^{1-2\gamma}\left(  \sup_{r\in\left[  0,h\right]  }E\left[
\left\vert S\left(  r\right)  \xi-\xi\right\vert ^{2}\right]  +h^{1-2\gamma
}\right)  . \label{conda}%
\end{align}
Next, we notice that
\[
E^{\mathcal{F}_{t}}\left[  q_{h}\right]  =E^{\mathcal{F}_{t}}\left[
\eta\right]  -S\left(  h\right)  \xi-\int_{t}^{t+h}S\left(  t+h_{n}-s\right)
f\left(  \xi,u(s)\right)  ds.
\]
Thus, using Jensen's inequality and (\ref{In10}), we get%
\begin{equation}
E\left[  \left\vert E^{\mathcal{F}_{t}}\left[  q_{h}\right]  \right\vert
^{2}\right]  \leq Ch^{2}\left(  h^{1-2\gamma}+\sup_{r\in\left[  0,h\right]
}E\left[  \left\vert S\left(  r\right)  \xi-\xi\right\vert ^{2}\right]
\right)  . \label{condb}%
\end{equation}
We introduce the $\mathcal{F}_{t+h}$-measurable random variable%
\begin{equation}
p_{h}=h^{\gamma-\frac{1}{2}}q_{h}. \label{ph}%
\end{equation}
The inequalities (\ref{conda}) and (\ref{condb}) imply%
\begin{equation}
E\left[  \left\vert p_{h}\right\vert ^{2}\right]  +\frac{1}{h^{1+2\gamma}%
}E\left[  \left\vert E^{\mathcal{F}_{t}}\left[  p_{h}\right]  \right\vert
^{2}\right]  \leq C\left(  \sup_{r\in\left[  0,h\right]  }E\left[  \left\vert
S\left(  r\right)  \xi-\xi\right\vert ^{2}\right]  +h^{1-2\gamma}\right)  .
\label{condp}%
\end{equation}
Using the strong continuity of $\left(  S(r)\right)  _{r\geq0}$ and a
dominated convergence argument, we get%
\[
\lim_{h\rightarrow0+}\left(  E\left[  \left\vert p_{h}\right\vert ^{2}\right]
+\frac{1}{h^{1+2\gamma}}E\left[  \left\vert E^{\mathcal{F}_{t}}\left[
p_{h}\right]  \right\vert ^{2}\right]  \right)  =0.
\]
Also, by the choice of $\eta$ and $p_{h},$
\begin{align*}
&  S\left(  h\right)  \xi+\int_{t}^{t+h}S\left(  t+h-s\right)  f\left(
\xi,u(s)\right)  ds\\
&  +\int_{t}^{t+h}S\left(  t+h-s\right)  g\left(  \xi,u(s)\right)
dW_{s}+h^{\left(  \frac{1}{2}-\gamma\right)  }p_{h}=\eta\in K,
\end{align*}
$dP-$almost surely$.$ The proof is now complete.
\end{proof}

\begin{remark}
1. A careful look at the previous proof shows that $\lim\inf$ in the
definition of quasi-tangency can be strengthen to $\lim\sup.$

2. Moreover, for every deterministic initial data $\xi=x\in K$ and every
$k\geq2,$ the random variable $p_{h}$ (given by (\ref{ph})) can be chosen in
$L^{k}\left(  \Omega,\mathcal{F}_{t+h},P;H\right)  $ whenever one of the
following assumptions holds true:

a) the set of constraints $K$ is viable with respect to (\ref{SDE1});

b) the set of constraints $K$ is convex and $\varepsilon$-viable;

c) the set $K$ is bounded and $\varepsilon$-viable.

Indeed, if a) holds true, then the random variable $\eta$ in (\ref{In0}) can
be chosen as $\eta=X_{t+h}^{t,x,u}.$ Under the assumption b), one can choose
$\eta=\pi_{K}\left(  X_{t+h}^{t,x,u}\right)  $, where $\pi_{K}$ is the
projection on $K$. If c) holds true, $\eta\in L^{\infty}\left(  \Omega
,\mathcal{F}_{t+h},P;H\right)  .$ The conclusion follows from (\ref{qh}) and
(\ref{ph}).
\end{remark}

\subsection{Sufficient condition for the $\varepsilon$-viability property}

In order to prove the converse, we introduce the notion of $\varepsilon
$-approximate mild solution.

\begin{definition}
For every $0\leq t\leq\widetilde{T}\leq T$, every initial condition $\xi\in
L^{2}\left(  \Omega,\mathcal{F}_{t},P;K\right)  $ and every positive real
constant $\varepsilon,$ an $\varepsilon$-approximate mild solution of
(\ref{SDE1}) defined on $\left[  t,\widetilde{T}\right]  $ is a sixtuple
$\left(  \sigma,u,\varphi,\psi,\theta,Y\right)  $ such that

(a) the function $\sigma:\left[  t,\widetilde{T}\right]  \longrightarrow
\left[  t,\widetilde{T}\right]  $ is non decreasing and such that
\[
s-\varepsilon\leq\sigma\left(  s\right)  \leq s,\text{ for all }s\in\left[
t,\widetilde{T}\right]  .
\]

(b) the process $u$ is an admissible control process.

(c) the process $\varphi:\left[  t,\widetilde{T}\right]  \longrightarrow H$ is
predictable and such that
\[
E\left[  \int_{t}^{\widetilde{T}}\left\vert \varphi\left(  s\right)
\right\vert ^{2}ds\right]  \leq\left(  \widetilde{T}-t\right)  \varepsilon.
\]

(d) the process $\psi:\left[  t,\widetilde{T}\right]  \longrightarrow
L_{2}\left(  \Xi;H\right)  $ is predictable and such that
\[
E\left[  \int_{t}^{\widetilde{T}}\left\vert \psi\left(  s\right)  \right\vert
_{L_{2}\left(  \Xi;H\right)  }^{2}ds\right]  \leq\left(  \widetilde
{T}-t\right)  \varepsilon.
\]

(e) the function $\theta:\left\{  \left(  s,r\right)  :t\leq r<s\leq
\widetilde{T}\right\}  \longrightarrow\left[  t,\widetilde{T}\right]  $ is
Lebesgue-measurable and it satisfies $\theta\left(  s,r\right)  \leq s-t$ and
$s\mapsto\theta\left(  s,r\right)  $ is nonexpansive on $\left(
r,\widetilde{T}\right]  $ (i.e. $\left\vert \theta\left(  s,r\right)
-\theta\left(  s^{\prime},r\right)  \right\vert \leq\left\vert s-s^{\prime
}\right\vert ,$ for every $s,s^{\prime}\in\left(  r,\widetilde{T}\right]  $).

(f) the process $Y:\left[  t,\widetilde{T}\right]  \longrightarrow H$ is
predictable. Moreover, the process $Y_{\sigma}:\left[  t,\widetilde{T}\right]
\longrightarrow H$ defined by
\[
Y_{\sigma}\left(  s\right)  =Y\left(  \sigma\left(  s\right)  \right)  ,\text{
for all }s\in\left[  t,\widetilde{T}\right]
\]
is predictable and
\begin{align*}
Y(s) &  =S\left(  s-t\right)  \xi+\int_{t}^{s}S\left(  s-r\right)  f\left(
Y\left(  \sigma\left(  r\right)  \right)  ,u\left(  r\right)  \right)  dr\\
\text{ \ \ \ \ } &  \text{\ \ \ }+\int_{t}^{s}S\left(  s-r\right)  g\left(
Y\left(  \sigma\left(  r\right)  \right)  ,u\left(  r\right)  \right)
dW_{r}\\
&  \text{ \ \ }+\int_{t}^{s}S\left(  \theta\left(  s,r\right)  \right)
\varphi\left(  r\right)  dr+\int_{t}^{s}S\left(  \theta\left(  s,r\right)
\right)  \psi\left(  r\right)  dW_{r},
\end{align*}
for all $s\in\left[  t,\widetilde{T}\right]  .$

(g) for every $s\in\left[  t,\widetilde{T}\right]  $, $Y\left(  \sigma\left(
s\right)  \right)  \in K,$ $dP$-almost surely and $Y\left(  \widetilde
{T}\right)  \in K,$ $dP$-almost surely. Moreover,%
\[
E\left[  \left\vert Y\left(  \sigma\left(  s\right)  \right)  -Y\left(
s\right)  \right\vert ^{2}\right]  \leq\varepsilon,\text{ for all }s\in\left[
t,\widetilde{T}\right]  \text{.}%
\]

\end{definition}

We begin by proving some qualitative properties of $\varepsilon$-approximate
mild solutions.

\begin{proposition}
We suppose that (\ref{A1}) and (\ref{A2}) hold true. If $t,\widetilde{T}%
\in\left[  0,T\right]  ,$ such that $t\leq\widetilde{T},$ $\xi\in L^{2}\left(
\Omega,\mathcal{F}_{t},P;K\right)  $, $\varepsilon\in\left(  0,1\right)  $ is
a positive real constant and $\left(  \sigma,u,\varphi,\psi,\theta,Y\right)  $
is an $\varepsilon$-approximate mild solution of (\ref{SDE1}) defined on
$\left[  t,\widetilde{T}\right]  $, then
\begin{equation}
\sup_{s\in\left[  t,\widetilde{T}\right]  }E\left[  \left\vert Y\left(
s\right)  \right\vert ^{2}\right]  \leq C.\label{In14}%
\end{equation}
Here $C$ is a positive real constant which only depends on $T$ and $\xi$ (but
not on $t,\widetilde{T},$ $\varepsilon$ nor on $\left(  \sigma,u,\varphi
,\psi,\theta,Y\right)  $)$.$
\end{proposition}

\begin{proof}
Let us fix $s\in\left[  t,\widetilde{T}\right]  $. In order to prove
(\ref{In14}), one uses the conditions (c) and (d) in Definition 3 to have
\begin{align}
E\left[  \left\vert Y\left(  s\right)  \right\vert ^{2}\right]   &  \leq
C\left(  E\left[  \left\vert S\left(  s-t\right)  \xi\right\vert ^{2}\right]
\right.  \nonumber\\
&  \text{ \ \ }+E\left[  \left(  \int_{t}^{s}\left\vert S\left(  s-r\right)
f\left(  Y\left(  \sigma\left(  r\right)  \right)  ,u\left(  r\right)
\right)  \right\vert dr\right)  ^{2}\right]  \nonumber\\
&  \text{ \ \ }\left.  +E\left[  \int_{t}^{s}\left\vert S\left(  s-r\right)
g\left(  Y\left(  \sigma\left(  r\right)  ,u\left(  r\right)  \right)
\right)  \right\vert ^{2}dr\right]  +\varepsilon\right)  \nonumber\\
&  =C\left(  I_{1}+I_{2}+I_{3}+\varepsilon).\right.  \label{In14.1}%
\end{align}
To estimate $I_{1}$, we use the properties of the semigroup $\left(  S\left(
r\right)  \right)  _{0\leq r\leq T}$ and obtain
\begin{equation}
I_{1}\leq CE\left[  \left\vert \xi\right\vert ^{2}\right]  .\label{In14.2}%
\end{equation}
For $I_{2},$ we write%
\begin{align}
I_{2} &  \leq CE\left[  \left(  \int_{t}^{s}\left\vert f\left(  Y\left(
\sigma\left(  r\right)  \right)  ,u\left(  r\right)  \right)  \right\vert
dr\right)  ^{2}\right]  \nonumber\\
&  \leq C\left(  s-t\right)  ^{1+2\gamma}\int_{t}^{s}\left(  s-r\right)
^{-2\gamma}E\left[  \left\vert f\left(  Y\left(  r\right)  ,u\left(  r\right)
\right)  \right\vert ^{2}\right]  dr\nonumber\\
&  +C\int_{t}^{s}E\left[  \left\vert Y\left(  r\right)  -Y\left(
\sigma\left(  r\right)  \right)  \right\vert ^{2}\right]  dr,\label{In14.3}%
\end{align}
Using property (g) in Definition 3 and the assumption (\ref{A1}), the
inequality (\ref{In14.3}) yields
\begin{equation}
I_{2}\leq C\left(  \int_{t}^{s}\left(  s-r\right)  ^{-2\gamma}E\left[
\left\vert Y\left(  r\right)  \right\vert ^{2}\right]  dr+1\right)
.\label{In14.4}%
\end{equation}
Similar arguments allow to obtain%
\begin{equation}
I_{3}\leq C\left(  \int_{t}^{s}\left(  \left(  s-r\right)  \wedge1\right)
^{-2\gamma}E\left[  \left\vert Y\left(  r\right)  \right\vert ^{2}\right]
dr+1\right)  .\label{In14.5}%
\end{equation}
We substitute (\ref{In14.2}), (\ref{In14.4}) and (\ref{In14.5}) in
(\ref{In14.1}) to finally get
\[
E\left[  \left\vert Y\left(  s\right)  \right\vert ^{2}\right]  \leq C\left(
1+\int_{t}^{s}\left(  s-r\right)  ^{-2\gamma}E\left[  \left\vert Y\left(
r\right)  \right\vert ^{2}\right]  dr\right)  ,
\]
for all $s\in\left[  t,\widetilde{T}\right]  .$ The conclusion follows from a
variant of Gronwall's inequality.
\end{proof}

The following result proves further regularity properties of the $Y$ component
of an approximate mild solution.

\begin{proposition}
If $t,\widetilde{T}\in\left[  0,T\right]  ,$ such that $t\leq\widetilde{T},$
the initial condition $\xi\in L^{2}\left(  \Omega,\mathcal{F}_{t},P;K\right)
$, $\varepsilon\in\left(  0,1\right)  $ is a positive real constant and
$\left(  \sigma,u,\varphi,\psi,\theta,Y\right)  $ is an $\varepsilon
$-approximate mild solution of (\ref{SDE1}), then $Y$ is mean-square continuous.
\end{proposition}

\begin{proof}
We let $\left(  \sigma,u,\varphi,\psi,\theta,Y\right)  $ be an $\varepsilon
$-approximate mild solution of (\ref{SDE1}) defined on $\left[  t,\widetilde
{T}\right]  .$ Let us fix $s\in\left[  t,\widetilde{T}\right]  $. For every
$s\leq s^{\prime},$%
\begin{align}
&  E\left[  \left\vert Y\left(  s^{\prime}\right)  -Y\left(  s\right)
\right\vert ^{2}\right]  \nonumber\\
&  \leq C\left(  E\left[  \left\vert S(s^{\prime}-s)\xi-\xi\right\vert
^{2}\right]  \right.  \nonumber\\
&  \text{ \ \ }+E\left[  \left\vert \int_{s}^{s^{\prime}}S\left(  s^{\prime
}-r\right)  f\left(  Y\left(  \sigma\left(  r\right)  \right)  ,u\left(
r\right)  \right)  dr\right\vert ^{2}\right]  \nonumber\\
&  \text{ \ \ }+E\left[  \left\vert \int_{t}^{s}\left(  S\left(  s^{\prime
}-r\right)  -S\left(  s-r\right)  \right)  f\left(  Y\left(  \sigma\left(
r\right)  \right)  ,u\left(  r\right)  \right)  dr\right\vert ^{2}\right]
\nonumber\\
&  \text{ \ \ }+E\left[  \left\vert \int_{s}^{s^{\prime}}S\left(  s^{\prime
}-r\right)  g\left(  Y\left(  \sigma\left(  r\right)  \right)  ,u\left(
r\right)  \right)  dW_{r}\right\vert ^{2}\right]  \nonumber\\
&  \text{ \ \ }+E\left[  \left\vert \int_{t}^{s}\left(  S\left(  s^{\prime
}-r\right)  -S\left(  s-r\right)  \right)  g\left(  Y\left(  \sigma\left(
r\right)  \right)  ,u\left(  r\right)  \right)  dW_{r}\right\vert ^{2}\right]
\nonumber\\
&  \text{ \ \ }+E\left[  \left\vert \int_{s}^{s^{\prime}}S\left(
\theta\left(  s^{\prime},r\right)  \right)  \varphi(r)dr\right\vert
^{2}+\left\vert \int_{s}^{s^{\prime}}S\left(  \theta\left(  s^{\prime
},r\right)  \right)  \psi\left(  r\right)  dW_{r}\right\vert ^{2}\right]
\nonumber\\
&  \text{ \ \ }+E\left[  \left\vert \int_{t}^{s}\left(  S\left(  \theta\left(
s^{\prime},r\right)  \right)  -S\left(  \theta\left(  s,r\right)  \right)
\right)  \varphi(r)dr\right\vert ^{2}\right]  \nonumber\\
&  \text{ \ \ }\left.  +E\left[  \left\vert \int_{t}^{s}\left(  S\left(
\theta\left(  s^{\prime},r\right)  \right)  -S\left(  \theta\left(
s,r\right)  \right)  \right)  \psi\left(  r\right)  dW_{r}\right\vert
^{2}\right]  \right)  \nonumber\\
&  =C\left(  I_{1}+I_{2}+I_{3}+I_{4}+I_{5}+I_{6}+I_{7}+I_{8}\right)
\label{In17}%
\end{align}
The strong continuity of the semigroup $S$ and a simple dominated convergence
argument yield%
\begin{equation}
\lim_{s^{\prime}\searrow s}I_{1}=0.\label{In17.1}%
\end{equation}
For the term $I_{2}$ (respectively $I_{4}$) we use (\ref{A1}) (respectively
(\ref{A2})) and Proposition 2 to get
\begin{equation}
\left\{
\begin{array}
[c]{c}%
I_{2}\leq C\left(  s^{\prime}-s\right)  ^{2}\text{ and}\\
I_{4}\leq C\left(  s^{\prime}-s\right)  ^{1-2\gamma}.
\end{array}
\right.  \label{In17.3}%
\end{equation}
For the term $I_{3}$ (respectively $I_{7}$) we use the continuity of the
semigroup $S$ and a dominated convergence argument to have%
\begin{equation}
\lim_{s^{\prime}\searrow s}I_{3}=0=\lim_{s^{\prime}\searrow s}I_{7}%
.\label{In17.5}%
\end{equation}
Next, one notices that
\[
I_{5}=E\left[  \int_{t}^{s}\left\vert \left(  S\left(  s^{\prime}-s\right)
-I\right)  S\left(  s-r\right)  g\left(  Y\left(  \sigma\left(  r\right)
\right)  ,u\left(  r\right)  \right)  \right\vert ^{2}dr\right]
\]
and, from the dominated convergence theorem it follows that
\begin{equation}
\lim_{s^{\prime}\searrow s}I_{5}=0.\label{In17.6}%
\end{equation}
Similar arguments hold true for $I_{8}$. Finally, the conditions (c) and (d)
in Definition 3 imply
\begin{equation}
\lim_{s^{\prime}\searrow s}I_{6}=0.\label{In17.7}%
\end{equation}
Combining (\ref{In17})-(\ref{In17.7}), we\ prove the mean-square
right-continuity of $Y.$ Similar arguments give the left-continuity. The proof
of the Proposition is now complete.
\end{proof}

The main step in the proof of Theorem 2 consists in the construction of
approximate mild solutions. To this purpose, we prove

\begin{lemma}
Let us suppose that (\ref{A1}) and (\ref{A2}) hold true and that $K\subset H$
is a nonempty, closed set which satisfies the quasi-tangency condition
(\ref{QTC1}) with $\lambda=0$. Then, for every $t\in\left[  0,T\right)  $,
every initial condition $\xi\in L^{2}\left(  \Omega,\mathcal{F}_{t}%
,P;K\right)  $, every time horizon $\widetilde{T}\in\left[  t,T\right]  $ and
for each $\varepsilon\in\left(  0,1\right)  ,$ there exists an $\varepsilon
$-approximate mild solution of (\ref{SDE1}) denoted by $\left(  \sigma
,u,\varphi,\psi,\theta,Y\right)  $ and defined on $\left[  t,\widetilde
{T}\right]  .$
\end{lemma}

\begin{proof}
Let us fix $t\in\left[  0,T\right)  $, $\xi\in L^{2}\left(  \Omega
,\mathcal{F}_{t},P;K\right)  $, $\widetilde{T}\in\left[  t,T\right]  $ and
$\varepsilon\in\left(  0,1\right)  .$ The proof of the Lemma will be given in
three steps.

\underline{Step 1}. We will first show the existence of an $\varepsilon
$-approximate mild solution on some small interval $\left[  t,t+\delta\right]
.$ We fix $\varepsilon^{\prime}\in\left(  0,\varepsilon\right)  .$ We will
latter specify how $\varepsilon^{\prime}$ should be chosen. \ Using the
quasi-tangency property of $K$, one gets the existence of some $\delta
\in\left(  0,\varepsilon^{\prime}\right)  $, of an admissible control process
$u$ and of a random variable $p\in L^{2}\left(  \Omega,\mathcal{F}_{t+\delta
},P;H\right)  $ such that
\[
E\left[  \left\vert p\right\vert ^{2}\right]  +\frac{1}{\delta}E\left[
\left\vert E^{\mathcal{F}_{t}}\left[  p\right]  \right\vert ^{2}\right]
\leq\varepsilon^{\prime}%
\]
and%
\begin{align}
&  S\left(  \delta\right)  \xi+\int_{t}^{t+\delta}S\left(  t+\delta-s\right)
f\left(  \xi,u(s)\right)  ds\label{ap2}\\
&  +\int_{t}^{t+\delta}S\left(  t+\delta-s\right)  g\left(  \xi,u(s)\right)
dW_{s}+\sqrt{\delta}p\in K,dP-\text{almost surely}.\nonumber
\end{align}
We define the functions $\sigma:\left[  t,t+\delta\right]  \longrightarrow
\left[  t,t+\delta\right]  $ by setting $\sigma\left(  s\right)  =t,$ for all
$s\in\left[  t,t+\delta\right]  $. Using the martingale representation theorem
for the random variable $p$, we get the existence of some $L_{2}\left(
\Xi;H\right)  $-valued predictable process $\eta$, defined on $\left[
t,t+\delta\right]  $ such that
\[
p=E^{\mathcal{F}_{t}}\left[  p\right]  +\int_{t}^{t+\delta}\eta_{s}%
dW_{s},\text{ }dP-a.s.
\]
We introduce $\varphi:$ $\left[  t,t+\delta\right]  \longrightarrow H$ given
by
\[
\varphi(s)=\frac{1}{\sqrt{\delta}}E^{\mathcal{F}_{t}}\left[  p\right]  ,\text{
for all }s\in\left[  t,t+\delta\right]  ,
\]
and $\psi:\left[  t,t+\delta\right]  \longrightarrow L_{2}\left(
\Xi;H\right)  $ given by%
\[
\psi\left(  s\right)  =\sqrt{\delta}\eta_{s},\text{ for all }s\in\left[
t,t+\delta\right]  .
\]
Moreover, we let $\theta:\left\{  t\leq r<s\leq t+\delta\right\}
\longrightarrow\left[  t,t+\delta\right]  $ be defined by%
\[
\theta\left(  s,r\right)  =0,\text{ for all }t\leq r<s\leq t+\delta.
\]
Next, we define a process $Y$ by%
\begin{align*}
Y(s) &  =S\left(  s-t\right)  \xi+\int_{t}^{s}S\left(  s-r\right)  f\left(
\xi,u\left(  r\right)  \right)  dr\\
&  \text{ \ \ }+\int_{t}^{s}S\left(  s-r\right)  g\left(  \xi,u\left(
r\right)  \right)  dW_{r}+\int_{t}^{s}\varphi\left(  r\right)  dr+\int_{t}%
^{s}\psi\left(  r\right)  dW_{r},
\end{align*}
for all $s\in\left[  t,t+\delta\right]  .$ We claim that $\left(
\sigma,u,\varphi,\psi,\theta,Y\right)  $ is an $\varepsilon$-approximate mild
solution. The conditions (a), (b), (e) and (f) of Definition 3 are obviously
satisfied. From the choice of $p$, one gets
\begin{align*}
&  E\left[  \left\vert p\right\vert ^{2}\right]  +\frac{1}{\delta}E\left[
\left\vert E^{\mathcal{F}_{t}}\left[  p\right]  \right\vert ^{2}\right]  \\
&  =E\left[  \left\vert E^{\mathcal{F}_{t}}\left[  p\right]  \right\vert
^{2}\right]  +E\left[  \int_{t}^{t+\delta}\left\vert \eta_{s}\right\vert
^{2}ds\right]  +\frac{1}{\delta}E\left[  \left\vert E^{\mathcal{F}_{t}}\left[
p\right]  \right\vert ^{2}\right]  \leq\varepsilon^{\prime}.
\end{align*}
Thus, the conditions (c) and (d) are also satisfied. Hence, we only need to
check the last condition of Definition 3. To this purpose, we recall that
$\delta<\varepsilon^{\prime}$ and write%
\begin{align}
E\left[  \left\vert Y(s)-\xi\right\vert ^{2}\right]   &  \leq C\left(
E\left[  \left\vert S\left(  s-t\right)  \xi-\xi\right\vert ^{2}\right]
\right.  \nonumber\\
&  \text{ \ \ }+E\left[  \left(  \int_{t}^{s}\left\vert S\left(  s-r\right)
f\left(  \xi,u\left(  r\right)  \right)  \right\vert dr\right)  ^{2}\right]
\nonumber\\
&  \text{ \ \ }+E\left[  \int_{t}^{s}\left\vert S\left(  s-r\right)  g\left(
\xi,u\left(  r\right)  \right)  \right\vert ^{2}dr\right]  \nonumber\\
&  \text{ \ \ }\left.  +E\left[  \left(  \int_{t}^{s}\left\vert \varphi\left(
r\right)  \right\vert dr\right)  ^{2}\right]  +E\left[  \int_{t}^{s}\left\vert
\psi\left(  r\right)  \right\vert ^{2}dr\right]  \right)  \nonumber\\
&  \leq C\left(  \sup_{r\in\left[  0,\varepsilon^{\prime}\right]  }E\left[
\left\vert S\left(  r\right)  \xi-\xi\right\vert ^{2}\right]  +I_{1}%
+I_{2}+\varepsilon^{\prime}\right)  ,\text{ }\label{In16}%
\end{align}
for all $s\in\left[  t,t+\delta\right]  .$ Using (\ref{A2}), we notice that
\[
I_{2}\leq C\left(  \varepsilon^{\prime}\right)  ^{1-2\gamma}\left(  E\left[
\left\vert \xi\right\vert ^{2}\right]  +1\right)  .
\]
Similar estimates hold true for $I_{1}.$ We return to (\ref{In16}) and get
\[
E\left[  \left\vert Y(s)-\xi\right\vert ^{2}\right]  \leq C\left(  \sup
_{r\in\left[  0,\varepsilon^{\prime}\right]  }E\left[  \left\vert S\left(
r\right)  \xi-\xi\right\vert ^{2}\right]  +\left(  \varepsilon^{\prime
}\right)  ^{1-2\gamma}\right)  ,
\]
for all $s\in\left[  t,t+\delta\right]  .$ The constant $C$ may be chosen to
depend only on $T$, $\xi$ and the Lipschitz constant of $f$ and $g$ (but not
on $\delta$, nor $\varepsilon^{\prime}$). Thus, by considering small enough
$\varepsilon^{\prime}\in\left(  0,\varepsilon\right)  $, we may assume,
without loss of generality, that
\[
E\left[  \left\vert Y(s)-\xi\right\vert ^{2}\right]  \leq\varepsilon,\text{
for all }s\in\left[  t,t+\delta\right]  .
\]
By the choice of $\sigma,$ $Y\left(  \sigma\left(  s\right)  \right)  =\xi\in
K,$ $dP$-almost surely, for all $s\in\left[  t,t+\delta\right]  $ and
(\ref{ap2}) implies that $Y\left(  t+\delta\right)  \in K,$ $dP$-almost
surely. It follows that the condition (g) is also satisfied and, thus,
$\left(  \sigma,u,\varphi,\psi,\theta,Y\right)  $ is an $\varepsilon
$-approximate mild solution.

\underline{Step 2}. To prove the existence of some approximate solution on the
whole interval $\left[  t,\widetilde{T}\right]  ,$ we use the following
result, also known as the Brezis-Browder Theorem:

\begin{theorem}
Let $\mathcal{S}$ be a nonempty set, $\preceq\subset\mathcal{S}\times
\mathcal{S}$ a preorder on $\mathcal{S}$ and let $\mathcal{N}:\mathcal{S}%
\longrightarrow%
\mathbb{R}
\cup\left\{  +\infty\right\}  $ be an increasing function. Suppose that each
increasing sequence in $\mathcal{S}$ is bounded from above. Then, for each
$a_{0}\in\mathcal{S}$, there exists an $\mathcal{N}$-maximal element $a^{\ast
}\in\mathcal{S}$ such that $a_{0}\preceq a^{\ast}.$
\end{theorem}

For proof of this result and further remarks, the reader is referred to
Theorem 2.1.1 in \cite{CNV} and references therein.

We now return to the proof of the Lemma. We introduce the set $\mathcal{S}$ of
all $\varepsilon$-approximate mild solutions defined on intervals of the form
$\left[  t,t+\alpha\right]  \subset\left[  t,\widetilde{T}\right]  .$ On this
set, we define the following preorder relation $\preceq:$ given two
$\varepsilon$-approximate mild solutions $\left(  \sigma_{1},u_{1},\varphi
_{1},\psi_{1},\theta_{1},Y_{1}\right)  $ defined on $\left[  t,t+\alpha
_{1}\right]  $, respectively $\left(  \sigma_{2},u_{2},\varphi_{2},\psi
_{2},\theta_{2},Y_{2}\right)  $ defined on $\left[  t,t+\alpha_{2}\right]  ,$
we write%
\[
\left(  \sigma_{1},u_{1},\varphi_{1},\psi_{1},\theta_{1},Y_{1}\right)
\preceq\left(  \sigma_{2},u_{2},\varphi_{2},\psi_{2},\theta_{2},Y_{2}\right)
\]
if $\alpha_{1}\leq\alpha_{2}$, $u_{1}=u_{2},$ $\varphi_{1}=\varphi_{2},$
$\psi_{1}=\psi_{2}$ on $\left[  t,t+\alpha_{1}\right]  \times\Omega$ up to an
evanescent set and $\theta_{1}=\theta_{2}$ almost everywhere on $\left\{
t\leq r<s\leq t+\alpha_{1}\right\}  .$ We consider an increasing arbitrary
sequence in $\mathcal{S}$
\[
\mathcal{L=}\left\{  \left(  \sigma_{n},u_{n},\varphi_{n},\psi_{n},\theta
_{n},Y_{n}\right)  \text{ defined respectively on }\left[  t,t+\alpha
_{n}\right]  ,\text{ }n\in%
\mathbb{N}
\right\}  .
\]
We define
\[
\alpha=\sup_{n}\alpha_{n}.
\]
If $\alpha=\alpha_{n}$ for some index, then the element $\left(  \sigma
_{n},u_{n},\varphi_{n},\psi_{n},Y_{n}\right)  $ is an upper bound for
$\mathcal{L}$. Otherwise, since $\sigma_{n}$ are increasing functions
satisfying (a), there exists the limit
\[
\lim_{n\rightarrow\infty}\uparrow\sigma_{n}\left(  \alpha_{n}\right)
\in\left[  t,t+\alpha\right]  .
\]
This allows to define an increasing function $\sigma:\left[  t,t+\alpha
\right]  $ $\rightarrow\left[  t,t+\alpha\right]  $ by setting
\[
\sigma(s)=\left\{
\begin{array}
[c]{l}%
\sigma_{n}\left(  s\right)  ,\text{ if }s\in\left[  t,t+\alpha_{n}\right]  ,\\
\lim_{n\rightarrow\infty}\uparrow\sigma_{n}\left(  t+\alpha_{n}\right)
,\text{ if }s=t+\alpha.
\end{array}
\right.
\]
The function $\sigma$ satisfies the condition (a) of Definition 3. We consider
an element $u^{0}\in U$ and introduce the control process
\[
u\left(  s\right)  =1_{\left[  t,t+\alpha_{n}\right]  }(s)u_{n}(s)+1_{\left\{
t+\alpha\right\}  }\left(  s\right)  u^{0},
\]
for $s\in\left[  t,t+\alpha\right]  .$ Next, we define
\[
\varphi\left(  s\right)  =\left\{
\begin{array}
[c]{l}%
\varphi_{n}\left(  s\right)  ,\text{ if }s\in\left[  t,t+\alpha_{n}\right]
,\\
0,\text{ if }s=t+\alpha.
\end{array}
\right.
\]
and
\[
\psi\left(  s\right)  =\left\{
\begin{array}
[c]{l}%
\psi_{n}\left(  s\right)  ,\text{ if }s\in\left[  t,t+\alpha_{n}\right]  ,\\
0,\text{ if }s=t+\alpha.
\end{array}
\right.
\]
For every $n$, one can extend $\varphi_{n}$ on $\left[  t,t+\alpha\right]  $
by setting
\[
\varphi_{n}\left(  s\right)  =0,\text{ for all }s\in\left(  t+\alpha
_{n},t+\alpha\right]  .
\]
Then $\varphi$ is the pointwise limit of $\varphi_{n}$ (except for an
evanescent set). Thus, $\varphi$ is predictable. Property (c) of $\varepsilon
$-approximate mild solutions yields
\[
E\left[  \int_{t}^{t+\alpha}\left\vert \varphi_{n}\left(  s\right)
\right\vert ^{2}ds\right]  \leq\alpha_{n}\varepsilon,\text{ for all }n\in%
\mathbb{N}
.
\]
Then, by Fatou's lemma, one gets
\begin{equation}
E\left[  \int_{t}^{t+\alpha}\left\vert \varphi\left(  s\right)  \right\vert
^{2}ds\right]  \leq\alpha\varepsilon,
\end{equation}
and the condition (c) holds for $\varphi$. Moreover, a simple dominated
convergence argument proves that $\varphi_{n}\rightarrow_{n}\varphi$ in
$L^{2}\left(  \Omega\times\left[  t,t+\alpha\right]  ;H\right)  $. The
condition (d) follows in the same way. Next, one notices that whenever $m,k\in%
\mathbb{N}
$ such that $m\geq k,$ for every $r\in\left[  t,t+\alpha_{k}\right)  $, the
nonexpansive property of $\theta_{m}$ and $\theta_{k}$ yields
\[
\left\vert \theta_{m}\left(  t+\alpha_{m},r\right)  -\theta_{k}\left(
t+\alpha_{k},r\right)  \right\vert \leq\left\vert \alpha_{m}-\alpha
_{k}\right\vert .
\]
This implies the existence of $\lim_{n}\theta_{n}\left(  t+\alpha
_{n},r\right)  .$ We define
\[
\theta\left(  s,r\right)  =\left\{
\begin{array}
[c]{l}%
\theta_{n}\left(  s,r\right)  ,\text{ if }t\leq r<s\leq t+\alpha_{n},\\
\lim_{n}\theta_{n}\left(  t+\alpha_{n},r\right)  ,\text{ if }t\leq
r<s=t+\alpha.
\end{array}
\right.
\]
We recall that
\begin{align*}
&  Y_{n}(t+\alpha_{n})\\
&  =S\left(  \alpha_{n}\right)  \xi+\int_{t}^{t+\alpha}1_{\left[
t,t+\alpha_{n}\right]  }\left(  r\right)  S\left(  t+\alpha_{n}-r\right)
f\left(  Y_{n}\left(  \sigma_{n}\left(  r\right)  \right)  ,u_{n}\left(
r\right)  \right)  dr\\
&  \text{ \ \ }+\int_{t}^{t+\alpha}1_{\left[  t,t+\alpha_{n}\right]  }\left(
r\right)  S\left(  t+\alpha_{n}-r\right)  g\left(  Y_{n}\left(  \sigma
_{n}\left(  r\right)  \right)  ,u_{n}\left(  r\right)  \right)  dW_{r}\\
&  \text{ \ \ }+\int_{t}^{t+\alpha}1_{\left[  t,t+\alpha_{n}\right]  }\left(
r\right)  S\left(  \theta\left(  t+\alpha_{n},r\right)  \right)  \varphi
_{n}\left(  r\right)  dr\\
&  \text{ \ \ }+\int_{t}^{t+\alpha}1_{\left[  t,t+\alpha_{n}\right]  }\left(
r\right)  S\left(  \theta\left(  t+\alpha_{n},r\right)  \right)  \psi
_{n}\left(  r\right)  dW_{r}.
\end{align*}
Proposition 2 and a simple dominated convergence argument allow to obtain the
existence of the limit
\[
\lim_{n}Y_{n}(t+\alpha_{n})\text{ in }L^{2}\left(  \Omega,\mathcal{F}%
_{t+\alpha},P;H\right)  .
\]
Moreover, since $K$ is closed, the limit is in $K$ $dP$-almost surely. We can
now define
\[
Y\left(  s\right)  =\left\{
\begin{array}
[c]{l}%
Y_{n}\left(  s\right)  ,\text{ if }s\in\left[  t,t+\alpha_{n}\right]  ,\\
\lim_{n}Y_{n}(t+\alpha_{n}),\text{ if }s=t+\alpha.
\end{array}
\right.
\]
We notice that, whenever $\sigma\left(  t+\alpha\right)  >\sigma\left(
t+\alpha_{n}\right)  ,$ for all $n\geq1$, by the mean-square continuity of
$Y$, the process $Y_{\sigma}$ can be seen as the pointwise limit of the
sequence
\[
Y_{n}\left(  \sigma_{n}\left(  \cdot\right)  \right)  1_{\left[
t,t+\alpha_{n}\right]  }\left(  \cdot\right)  +Y\left(  \sigma_{n}\left(
t+\alpha_{n}\right)  \right)  1_{\left(  t+\alpha_{n},t+\alpha\right]
}\left(  \cdot\right)  .
\]
Thus, $Y_{\sigma}$ is predictable.

Let us check the condition (g) of Definition 3. We need to show that $Y\left(
\sigma\left(  s\right)  \right)  \in K$, $dP-a.s.$ and for all $s\in\left[
t,t+\alpha\right]  .$ If $s\leq t+\alpha_{n}$ for some $n$, then $Y\left(
\sigma\left(  s\right)  \right)  \in K,$ $dP-a.s.$ Otherwise$,$ using the fact
that $\sigma\left(  t+\alpha\right)  =\lim_{n}\sigma_{n}\left(  t+\alpha
_{n}\right)  $, $\ $and $Y$ is mean-square continuous, we get $Y\left(
\sigma\left(  t+\alpha\right)  \right)  \in K,$ $dP-a.s$. In order to prove
that $\left(  \sigma,u,\varphi,\psi,\theta,Y\right)  $ is an $\varepsilon
$-approximate mild solution on $\left[  t,t+\alpha\right]  $ one only needs to
verify%
\[
E\left[  \left\vert Y\left(  s\right)  -Y\left(  \sigma\left(  s\right)
\right)  \right\vert ^{2}\right]  \leq\varepsilon,
\]
for all $s\in\left[  t,t+\alpha\right]  $. If $s\leq t+\alpha_{n}$ for some
$n$, we have nothing to prove. We recall that
\[
E\left[  \left\vert Y\left(  t+\alpha_{n}\right)  -Y\left(  \sigma\left(
t+\alpha_{n}\right)  \right)  \right\vert ^{2}\right]  \leq\varepsilon,
\]
for all $n\geq1$. Using the definition of $Y$ and $\sigma$ and the continuity
of $Y,$ we also get%
\[
E\left[  \left\vert Y\left(  t+\alpha\right)  -Y\left(  \sigma\left(
t+\alpha\right)  \right)  \right\vert ^{2}\right]  \leq\varepsilon.
\]
It follows that $\left(  \sigma,u,\varphi,\psi,\theta,Y\right)  $ is an
$\varepsilon$-approximate mild solution on $\left[  t,t+\alpha\right]  $ and
an upper bound for $\mathcal{L}$. We introduce the increasing function
\[
\mathcal{N}:\mathcal{S}\longrightarrow%
\mathbb{R}
_{+},\text{ given by }\mathcal{N}\left(  \left(  \sigma,u,\varphi,\psi
,\theta,Y\right)  \right)  =\alpha,\text{\ }%
\]
whenever $\left(  \sigma,u,\varphi,\psi,\theta,Y\right)  $ is defined on
$\left[  t,t+\alpha\right]  .$ We apply the Brezis-Browder Theorem to obtain
the existence of an $\mathcal{N}$-maximal element of $\mathcal{S}$ denoted by
$\left(  \sigma^{\ast},u^{\ast},\varphi^{\ast},\psi^{\ast},Y^{\ast}\right)  $
and defined on $\left[  t,t+\alpha^{\ast}\right]  .$

\underline{Step 3}. We claim that $t+\alpha^{\ast}=\widetilde{T}.$ Let us
assume, for the moment, that $t+\alpha^{\ast}<\widetilde{T}.$ By definition,
$Y^{\ast}(t+\alpha^{\ast})\in K$ $dP$-a.s. We recall that $K$ satisfies the
quasi-tangency condition with respect to the control system (\ref{SDE1}).
Therefore, for every $\varepsilon^{\prime}<\varepsilon$ there exist
$0<\delta^{\ast}\leq\min\left\{  \widetilde{T}-t-\alpha^{\ast},\varepsilon
^{\prime}\right\}  $, $p^{\ast}\in L^{2}\left(  \Omega,\mathcal{F}%
_{t+\alpha^{\ast}+\delta^{\ast}},P;H\right)  $ and an admissible control
process $\widetilde{u}\in\mathcal{A}$ such that
\begin{equation}
E\left[  \left\vert p^{\ast}\right\vert ^{2}\right]  +\frac{1}{\delta^{\ast}%
}E\left[  \left\vert E^{\mathcal{F}_{t+\alpha^{\ast}}}\left[  p^{\ast}\right]
\right\vert ^{2}\right]  \leq\varepsilon^{\prime},\text{ and}\label{In2}%
\end{equation}%
\begin{align}
&  S\left(  \delta^{\ast}\right)  Y^{\ast}\left(  t+\alpha^{\ast}\right)
+\int_{t+\alpha^{\ast}}^{t+\alpha^{\ast}+\delta^{\ast}}S\left(  t+\alpha
^{\ast}+\delta^{\ast}-s\right)  f\left(  Y^{\ast}\left(  t+\alpha^{\ast
}\right)  ,\widetilde{u}(s)\right)  ds\nonumber\\
&  +\int_{t+\alpha^{\ast}}^{t+\alpha^{\ast}+\delta^{\ast}}S\left(
t+\alpha^{\ast}+\delta^{\ast}-s\right)  g\left(  Y^{\ast}\left(
t+\alpha^{\ast}\right)  ,\widetilde{u}(s)\right)  dW_{s}+\sqrt{\delta^{\ast}%
}p^{\ast}\in K,\label{ap1}%
\end{align}
$dP-$almost surely$.$ The martingale representation theorem yields the
existence of some predictable process $\eta^{\ast}$ defined on $\left[
t+\alpha^{\ast},t+\alpha^{\ast}+\delta^{\ast}\right]  $ such that
\[
p^{\ast}=E^{\mathcal{F}_{t+\alpha^{\ast}}}\left[  p^{\ast}\right]
+\int_{t+\alpha^{\ast}}^{t+\alpha^{\ast}+\delta^{\ast}}\eta_{s}^{\ast}dW_{s}.
\]
We introduce the functions%
\begin{equation}
\sigma(s)=\left\{
\begin{array}
[c]{l}%
\sigma^{\ast}\left(  s\right)  ,\text{ if }s\in\left[  t,t+\alpha^{\ast
}\right]  ,\\
t+\alpha^{\ast},\text{ if }s\in\left(  t+\alpha^{\ast},t+\alpha^{\ast}%
+\delta^{\ast}\right]  ,
\end{array}
\right.  \label{sigma}%
\end{equation}%
\[
u\left(  s\right)  =\left\{
\begin{array}
[c]{l}%
u^{\ast}\left(  s\right)  ,\text{ if }s\in\left[  t,t+\alpha^{\ast}\right]
,\\
\widetilde{u}\left(  s\right)  ,\text{ if }s\in\left(  t+\alpha^{\ast
},t+\alpha^{\ast}+\delta^{\ast}\right]  ,
\end{array}
\right.
\]%
\begin{equation}
\varphi\left(  s\right)  =\left\{
\begin{array}
[c]{l}%
\varphi^{\ast}\left(  s\right)  ,\text{ if }s\in\left[  t,t+\alpha^{\ast
}\right]  ,\\
\frac{1}{\sqrt{\delta^{\ast}}}E^{\mathcal{F}_{t+\alpha^{\ast}}}\left[
p^{\ast}\right]  ,\text{ if }s\in\left(  t+\alpha^{\ast},t+\alpha^{\ast
}+\delta^{\ast}\right]  ,
\end{array}
\right.  \label{phi}%
\end{equation}%
\[
\psi\left(  s\right)  =\left\{
\begin{array}
[c]{l}%
\psi^{\ast}\left(  s\right)  ,\text{ if }s\in\left[  t,t+\alpha^{\ast}\right]
,\\
\sqrt{\delta^{\ast}}\eta_{s}^{\ast},\text{ if }s\in\left(  t+\alpha^{\ast
},t+\alpha^{\ast}+\delta^{\ast}\right]  ,
\end{array}
\right.
\]%
\[
\theta\left(  s,r\right)  =\left\{
\begin{array}
[c]{l}%
\theta^{\ast}\left(  s,r\right)  ,\text{ if }t\leq r<s\leq t+\alpha^{\ast},\\
\left(  s-t-\alpha^{\ast}\right)  +\theta^{\ast}\left(  t+\alpha^{\ast
},r\right)  ,\\
\text{ \ \ \ \ \ \ \ \ \ \ \ \  if }t<r<t+\alpha^{\ast}<s\leq t+\alpha^{\ast
}+\delta^{\ast},\\
0,\text{ \ \ \ \ \ \ \ \ \ if }t+\alpha^{\ast}\leq r<s\leq t+\alpha^{\ast
}+\delta^{\ast}%
\end{array}
\right.
\]
and%
\[
Y\left(  s\right)  =\left\{
\begin{array}
[c]{l}%
Y^{\ast}\left(  s\right)  ,\text{ if }s\in\left[  t,t+\alpha^{\ast}\right]
,\\
S(s-t-\alpha^{\ast})Y^{\ast}\left(  t+\alpha^{\ast}\right)  +\int
_{t+\alpha^{\ast}}^{s}S\left(  s-r\right)  f\left(  Y^{\ast}\left(
t+\alpha^{\ast}\right)  ,u\left(  r\right)  \right)  dr\\
\multicolumn{1}{r}{+\int_{t+\alpha^{\ast}}^{s}S\left(  s-r\right)  g\left(
Y^{\ast}\left(  t+\alpha^{\ast}\right)  ,u\left(  r\right)  \right)
dW_{r}+\int_{t+\alpha^{\ast}}^{s}\varphi\left(  r\right)  dr}\\
\text{ \ \ \ \ \ \ \ \ \ \ }+\int_{t+\alpha^{\ast}}^{s}\psi\left(  r\right)
dW_{r},\text{ if }s\in\left(  t+\alpha^{\ast},t+\alpha^{\ast}+\delta^{\ast
}\right]  .
\end{array}
\right.
\]
It suffices now to choose $\varepsilon^{\prime}$ as in Step 1 to prove that
$\left(  \sigma,u,\varphi,\psi,\theta,Y\right)  $ is an element of
$\mathcal{S}$. Moreover,
\begin{align*}
\left(  \sigma^{\ast},u^{\ast},\varphi^{\ast},\psi^{\ast},\theta^{\ast
},Y^{\ast}\right)   &  \preceq\left(  \sigma,u,\varphi,\psi,\theta,Y\right)
\text{ and}\\
\mathcal{N}\left(  \left(  \sigma^{\ast},u^{\ast},\varphi^{\ast},\psi^{\ast
},\theta^{\ast},Y^{\ast}\right)  \right)   &  <\mathcal{N}\left(  \left(
\sigma,u,\varphi,\psi,\theta,Y\right)  \right)  .
\end{align*}
This inequalities come in contradiction with the initial assumption of

$\left(  \sigma^{\ast},u^{\ast},\varphi^{\ast},\psi^{\ast},Y^{\ast}\right)  $
being maximal. We deduce that $\left(  \sigma^{\ast},u^{\ast},\varphi^{\ast
},\psi^{\ast},Y^{\ast}\right)  $ is an $\varepsilon-$approximate mild solution
defined on $\left[  t,\widetilde{T}\right]  $ and this completes the proof of
our Lemma.
\end{proof}

Using this result, we are now able to prove that the quasi-tangency condition
(\ref{QTC1}) written for $\lambda=0$ provides a sufficient condition for
$\varepsilon$-viability.

\begin{proof}
(of Theorem 2). We assume that $K$ satisfies the quasi-tangency condition
(\ref{QTC1}) with $\lambda=0$. Let us fix $t\in\left[  0,T\right)  $, $\xi\in
L^{2}\left(  \Omega,\mathcal{F}_{t},P;K\right)  $ and $\varepsilon\in\left(
0,1\right)  .$ We apply the previous Lemma and get the existence of an
$\varepsilon$-approximate mild solution of (\ref{SDE1}) denoted $\left(
\sigma,u,\varphi,\psi,\theta,Y\right)  $ which is defined on $\left[
t,T\right]  .$ From the definition of $\varepsilon$-approximate mild
solutions,%
\begin{align}
Y(s) &  =S\left(  s-t\right)  \xi+\int_{t}^{s}S\left(  s-r\right)  f\left(
Y\left(  \sigma\left(  r\right)  \right)  ,u\left(  r\right)  \right)
dr\nonumber\\
&  +\int_{t}^{s}S\left(  s-r\right)  g\left(  Y\left(  \sigma\left(  r\right)
\right)  ,u\left(  r\right)  \right)  dW_{r}\nonumber\\
&  +\int_{t}^{s}S\left(  \theta\left(  s,r\right)  \right)  \varphi\left(
r\right)  dr+\int_{t}^{s}S\left(  \theta\left(  s,r\right)  \right)
\psi\left(  r\right)  dW_{r},\label{In3}%
\end{align}
$dP-a.s.,$ for all $s\in\left[  t,T\right]  $, $Y\left(  \sigma\left(
s\right)  \right)  \in K,$ $dP$-almost surely, for all $s\in\left[
t,T\right]  $ and $Y\left(  T\right)  \in K,$ $dP$-almost surely. Moreover,%
\begin{equation}
E\left[  \left\vert Y\left(  \sigma\left(  s\right)  \right)  -Y\left(
s\right)  \right\vert ^{2}\right]  \leq\varepsilon,\label{In6}%
\end{equation}
for all $s\in\left[  t,T\right]  $. It follows that
\begin{align}
&  d^{2}\left(  X^{t,\xi,u}\left(  s\right)  ,L^{2}\left(  \Omega
,\mathcal{F}_{s},P;K\right)  \right)  \leq E\left[  \left\vert X^{t,\xi
,u}\left(  s\right)  -Y\left(  \sigma\left(  s\right)  \right)  \right\vert
^{2}\right]  \nonumber\\
&  \leq2\left(  E\left[  \left\vert Y\left(  s\right)  -Y\left(  \sigma\left(
s\right)  \right)  \right\vert ^{2}\right]  +E\left[  \left\vert X^{t,\xi
,u}\left(  s\right)  -Y\left(  s\right)  \right\vert ^{2}\right]  \right)
,\label{In4}%
\end{align}
for all $s\in\left[  t,T\right]  .$ Next, in order to estimate $E\left[
\left\vert X^{t,\xi,u}\left(  s\right)  -Y\left(  s\right)  \right\vert
^{2}\right]  ,$ we use%
\begin{align}
&  E\left[  \left\vert X^{t,\xi,u}\left(  s\right)  -Y\left(  s\right)
\right\vert ^{2}\right]  \nonumber\\
&  \leq C\left(  E\left[  \left\vert \int_{t}^{s}S\left(  s-r\right)  \left(
f\left(  Y\left(  r\right)  ,u\left(  r\right)  \right)  -f\left(  X^{t,\xi
,u}\left(  r\right)  ,u\left(  r\right)  \right)  \right)  dr\right\vert
^{2}\right]  \right.  \nonumber\\
&  \text{ \ \ }+E\left[  \left\vert \int_{t}^{s}S\left(  s-r\right)  \left(
f\left(  Y\left(  r\right)  ,u\left(  r\right)  \right)  -f\left(  Y\left(
\sigma\left(  r\right)  \right)  ,u\left(  r\right)  \right)  \right)
dr\right\vert ^{2}\right]  \nonumber\\
&  \text{ \ \ }+E\left[  \left\vert \int_{t}^{s}S\left(  s-r\right)  \left(
g\left(  Y\left(  r\right)  ,u\left(  r\right)  \right)  -g\left(  X^{t,\xi
,u}\left(  r\right)  ,u\left(  r\right)  \right)  \right)  dW_{r}\right\vert
^{2}\right]  \nonumber\\
&  \text{ \ \ }+E\left[  \left\vert \int_{t}^{s}S\left(  s-r\right)  \left(
f\left(  Y\left(  r\right)  ,u\left(  r\right)  \right)  -f\left(  Y\left(
\sigma\left(  r\right)  \right)  ,u\left(  r\right)  \right)  \right)
dW_{r}\right\vert ^{2}\right]  \nonumber\\
&  \text{ \ \ }\left.  +E\left[  \left\vert \int_{t}^{s}S\left(  \theta\left(
s,r\right)  \right)  \varphi\left(  s\right)  ds\right\vert ^{2}\right]
+E\left[  \left\vert \int_{t}^{s}S\left(  \theta\left(  s,r\right)  \right)
\psi\left(  s\right)  dW_{s}\right\vert ^{2}\right]  \right)  \nonumber\\
&  =C\left(  I_{1}+I_{2}+I_{3}+I_{4}+I_{5}+I_{6}\right)  ,\label{In5}%
\end{align}
for all $s\in\left[  t,T\right]  .$ In order to estimate $I_{1},$ we use the
Cauchy-Schwarz inequality and the Lipschitz property of $f,$ and get%
\begin{align}
I_{1} &  \leq CE\left[  \left(  \int_{t}^{s}\left\vert Y\left(  r\right)
-X^{t,\xi,u}\left(  r\right)  \right\vert dr\right)  ^{2}\right]  \nonumber\\
&  \leq C\left(  s-t\right)  ^{1+2\gamma}\int_{t}^{s}\left(  s-r\right)
^{-2\gamma}E\left[  \left\vert Y\left(  r\right)  -X^{t,\xi,u}\left(
r\right)  \right\vert ^{2}\right]  dr.\label{In7}%
\end{align}
For $I_{2},$ similar arguments combined with (\ref{In6}) yield%
\begin{equation}
I_{2}\leq C\varepsilon.\label{In8}%
\end{equation}
The assumption (A2) gives%
\begin{align}
I_{3} &  =E\left[  \int_{t}^{s}\left\vert S\left(  s-r\right)  \left(
g\left(  Y\left(  r\right)  ,u\left(  r\right)  \right)  -g\left(  X^{t,\xi
,u}\left(  r\right)  ,u\left(  r\right)  \right)  \right)  \right\vert
^{2}dr\right]  \nonumber\\
&  \leq C\int_{t}^{s}\left(  \left(  s-r\right)  \wedge1\right)  ^{-2\gamma
}E\left[  \left\vert Y\left(  r\right)  -X^{t,\xi,u}\left(  r\right)
\right\vert ^{2}\right]  dr\label{In9}%
\end{align}
and%
\begin{equation}
I_{4}\leq C\varepsilon.\label{In11}%
\end{equation}
For the last terms $I_{5,6}$ it suffices to recall properties (c) and (d) of
Definition 3 and get
\begin{equation}
I_{5}+I_{6}\leq C\varepsilon.\label{In12}%
\end{equation}
We substitute the estimates (\ref{In7})-(\ref{In12}) in (\ref{In5}) to obtain%
\begin{align*}
&  E\left[  \left\vert X^{t,\xi,u}\left(  s\right)  -Y\left(  s\right)
\right\vert ^{2}\right]  \\
&  \leq C\left(  \varepsilon+\int_{t}^{s}\left(  s-r\right)  ^{-2\gamma
}E\left[  \left\vert Y\left(  r\right)  -X^{t,\xi,u}\left(  r\right)
\right\vert ^{2}\right]  dr\right)  ,
\end{align*}
for all $s\in\left[  t,T\right]  .$ Then, applying a variant of Gronwall's
inequality, we have
\begin{equation}
E\left[  \left\vert X^{t,\xi,u}\left(  s\right)  -Y\left(  s\right)
\right\vert ^{2}\right]  \leq C\varepsilon,\label{In13}%
\end{equation}
for all $s\in\left[  t,T\right]  .$ We substitute (\ref{In6}) and (\ref{In13})
in (\ref{In4}) to finally get
\[
d^{2}\left(  X^{t,\xi,u}\left(  s\right)  ,L^{2}\left(  \Omega,\mathcal{F}%
_{s},P;K\right)  \right)  \leq C\varepsilon,
\]
for all $s\in\left[  t,T\right]  .$ The conclusion follows by recalling that
$C$ can be chosen independent of $\varepsilon$ (and $s\in\left[  t,T\right]  $
) and $\varepsilon\in\left(  0,1\right)  $ is arbitrary. The proof of the main
result is now complete.
\end{proof}

\section{Applications}

\subsection{The linear case}

Let us now consider the following particular case of the control system
(\ref{SDE1}):%
\begin{equation}
\left\{
\begin{array}
[c]{l}%
dX^{t,\xi,u}\left(  s\right)  =\left(  AX^{t,\xi,u}\left(  s\right)
+Bu\left(  s\right)  \right)  ds+\left(  CX^{t,\xi,u}\left(  s\right)
+Du\left(  s\right)  \right)  dW_{s},\\
\text{for }s\in\left[  t,T\right]  ,\\
X^{t,\xi,u}\left(  t\right)  =\xi\in L^{2}\left(  \Omega,\mathcal{F}%
_{t},P;H\right)  .
\end{array}
\right.  \label{SDE2}%
\end{equation}
Here $A$ is a linear unbounded operator on $H$ that generates a $C_{0}%
$-semigroup of linear operators $\left(  S\left(  t\right)  \right)  _{t\geq
0},$ $B\in\mathcal{L}\left(  G;H\right)  ,$ $C$ is an $\mathcal{L}\left(
\Xi;H\right)  $-valued linear operator on $H$ and there exist $\gamma
\in\left[  0,\frac{1}{2}\right)  $ and $c>0$ such that, for every $t>0,$
$S\left(  t\right)  C\in\mathcal{L}\left(  H;L_{2}\left(  \Xi;H\right)
\right)  $ and $\left\vert S\left(  t\right)  C\right\vert _{\mathcal{L}%
\left(  H;L_{2}\left(  \Xi;H\right)  \right)  }\leq c\left(  t\wedge1\right)
^{-\gamma},$ and $D$ is an $L_{2}\left(  \Xi;H\right)  $-valued linear bounded
operator on $G.$ We also suppose that
\begin{equation}
U\text{ is a closed, bounded and convex subset of }G. \tag{A3}\label{A3}%
\end{equation}

\begin{remark}
If the assumption (\ref{A3}) holds true, the space of admissible control
processes $\mathcal{A}$ is convex. As a consequence, $\mathcal{A}$ is a closed
subspace of $L^{2}\left(  \left[  t,T\right]  \times\Omega;G\right)  $ with
respect to the weak topology on $L^{2}\left(  \left[  t,T\right]  \times
\Omega;G\right)  $.
\end{remark}

It is obvious that viability implies $\varepsilon$-viability for a closed set
$K\subset H$. For the particular case of a linear control system we will prove
that the quasi-tangency condition written for $\lambda=0$ is a sufficient
condition not only for the $\varepsilon-$viability, but also for the viability
property of an arbitrary nonempty, closed and convex set $K\subset H.$ Hence,
whenever $C\in\mathcal{L}\left(  H;L_{2}\left(  \Xi;H\right)  \right)  ,$
viability and $\varepsilon$-viability of closed, convex sets with respect to
the linear control system (\ref{SDE2}) are equivalent. The main result of this
section is

\begin{theorem}
Let us suppose that (\ref{A1}), (\ref{A2}) and (\ref{A3}) hold true. Moreover,
we suppose that $K\subset H$ is a nonempty, closed and convex set that
satisfies the quasi-tangency condition with respect to the control system
(\ref{SDE2}) with $\lambda=0$. Then $K$ is viable with respect to the control
system (\ref{SDE2}).
\end{theorem}

\begin{proof}
Let us fix $t\in\left[  0,T\right)  $ and $\xi\in L^{2}\left(  \Omega
,\mathcal{F}_{t},P;K\right)  .$ For every $n\geq2$, Lemma 1 gives the
existence of an $n^{-1}$-approximate mild solution denoted by $\left(
\sigma_{n},u_{n},\varphi_{n},\psi_{n},Y_{n}\right)  $ and defined on $\left[
t,T\right]  .$ By definition, for every $n\geq2$ and every $s\in\left[
t,T\right]  ,$%
\begin{align}
Y_{n}(s) &  =S\left(  s-t\right)  \xi+\int_{t}^{s}S\left(  s-r\right)
Bu_{n}\left(  r\right)  dr+\int_{t}^{s}S\left(  s-r\right)  CY_{n}\left(
\sigma_{n}\left(  r\right)  \right)  dW_{r}\nonumber\\
&  \text{ \ \ }+\int_{t}^{s}S\left(  s-r\right)  Du_{n}\left(  r\right)
dW_{r}\nonumber\\
&  \text{ \ \ }+\int_{t}^{s}S\left(  \theta_{n}\left(  s,r\right)  \right)
\varphi_{n}\left(  r\right)  dr+\int_{t}^{s}S\left(  \theta_{n}\left(
s,r\right)  \right)  \psi_{n}\left(  r\right)  dW_{r},\label{Yn}%
\end{align}
$dP$-almost surely. The estimates of Proposition 2 yield
\[
\sup_{n\geq2}\sup_{s\in\left[  t,T\right]  }E\left[  \left\vert Y_{n}\left(
s\right)  \right\vert ^{2}\right]  \leq c,
\]
for a generic constant $c$. Moreover, since $U$ is bounded,
\[
\sup_{n\geq2}\sup_{s\in\left[  t,T\right]  }E\left[  \left\vert u_{n}\left(
s\right)  \right\vert ^{2}\right]  \leq c.
\]
The above estimates, together with the assumption (\ref{A3}), allow to find a
subsequence (still denoted by $\left(  Y_{n},u_{n}\right)  $), a process $Y\in
L^{2}\left(  \left[  t,T\right]  ;L^{2}\left(  \Omega;H\right)  \right)  $ and
an admissible control process $u$ such that $\left(  Y_{n},u_{n}\right)
\rightarrow\left(  Y,u\right)  $ in the weak topology on $L^{2}\left(  \left[
t,T\right]  ;L^{2}\left(  \Omega;H\times U\right)  \right)  .$

\underline{Step 1}. We begin by showing that $Y$ can be identified with
$X^{t,\xi,u}$ (the unique mild solution of (\ref{SDE2}) starting from $\xi$
$\ $and associated to the control process $u$)$.$ We make the following
notations%
\begin{align*}
&  M_{s}^{1,n}=\int_{t}^{s}S\left(  s-r\right)  Bu_{n}\left(  r\right)
dr,\text{ }M_{s}^{1}=\int_{t}^{s}S\left(  s-r\right)  Bu\left(  r\right)
dr,\\
&  M_{s}^{2,n}=\int_{t}^{s}S\left(  s-r\right)  CY_{n}\left(  \sigma
_{n}\left(  r\right)  \right)  dW_{r},\text{ }M_{s}^{2}=\int_{t}^{s}S\left(
s-r\right)  CY\left(  r\right)  dW_{r},\text{ }\\
&  M_{s}^{3,n}=\int_{t}^{s}S\left(  s-r\right)  Du_{n}\left(  r\right)
dW_{r},\text{ }M_{s}^{3}=\int_{t}^{s}S\left(  s-r\right)  Du\left(  r\right)
dW_{r},\\
&  M_{s}^{4,n}=\int_{t}^{s}S\left(  \theta_{n}\left(  s,r\right)  \right)
\varphi_{n}\left(  r\right)  dr+\int_{t}^{s}S\left(  \theta_{n}\left(
s,r\right)  \right)  \psi_{n}\left(  r\right)  dW_{r},\text{ for all }%
s\in\left[  t,T\right]  .
\end{align*}
Let us fix $\phi\in L^{2}\left(  \Omega,\mathcal{F}_{s},P;H\right)  .$ Using
the weak convergence of $\left(  u_{n}\right)  ,$ one gets
\begin{align*}
&  \lim_{n\rightarrow\infty}E\left[  \left\langle M_{s}^{1,n},\phi
\right\rangle \right]  =\lim_{n\rightarrow\infty}E\left[  \int_{t}%
^{s}\left\langle S\left(  s-r\right)  Bu_{n}\left(  r\right)  ,\phi
\right\rangle \right]  dr\\
&  =\lim_{n\rightarrow\infty}E\left[  \int_{t}^{s}\left\langle u_{n}\left(
r\right)  ,B^{\ast}S^{\ast}\left(  s-r\right)  \phi\right\rangle dr\right]  \\
&  =E\left[  \int_{t}^{s}\left\langle u\left(  r\right)  ,B^{\ast}S^{\ast
}\left(  s-r\right)  \phi\right\rangle dr\right]  \\
&  =E\left[  \left\langle M_{s}^{1},\phi\right\rangle \right]  .
\end{align*}
If $\Phi\in L^{2}\left(  \left[  t,T\right]  ;L^{2}\left(  \Omega;H\right)
\right)  $ is an $\left(  \mathcal{F}_{t}\right)  $-adapted process$,$ the
previous equality, combined with a dominated convergence argument, allows to
prove that%
\[
\lim_{n}E\left[  \int_{t}^{T}\left\langle M_{s}^{1,n},\Phi\left(  s\right)
\right\rangle ds\right]  =E\left[  \int_{t}^{T}\left\langle M_{s}^{1}%
,\Phi\left(  s\right)  \right\rangle ds\right]  .
\]
Since $\Phi$ is arbitrary, this proves that $M^{1,n}$ converges in the weak
topology on $L^{2}\left(  \left[  t,T\right]  ;L^{2}\left(  \Omega;H\right)
\right)  $ to $M^{1}.$ Using condition (g) in the Definition 3 of approximate
mild solutions, we get
\[
\sup_{s\in\left[  t,T\right]  }E\left[  \left\vert Y_{n}\left(  s\right)
-Y_{n}\left(  \sigma_{n}\left(  s\right)  \right)  \right\vert ^{2}\right]
\leq n^{-1}.
\]
Thus, in order to prove that $M^{2,n}$ converges in the weak topology on
$L^{2}\left(  \left[  t,T\right]  ;L^{2}\left(  \Omega;H\right)  \right)  $ to
$M^{2},$ one can replace $M^{2,n}$ by the process $N^{2,n}$ given by
\[
N_{s}^{2,n}=\int_{t}^{s}S\left(  s-r\right)  CY_{n}\left(  r\right)
dW_{r},\text{ for all }s\in\left[  t,T\right]  .
\]
If $\phi\in L^{2}\left(  \Omega,\mathcal{F}_{s},P;H\right)  ,$ using the
martingale representation theorem we prove
\[
\lim_{n\rightarrow\infty}E\left[  \left\langle N_{s}^{2,n},\phi\right\rangle
\right]  =E\left[  \left\langle M_{s}^{2},\phi\right\rangle \right]  .
\]
Using, as before, the dominated convergence theorem, we get that $N^{2,n}$
converges in the weak topology on $L^{2}\left(  \left[  t,T\right]
;L^{2}\left(  \Omega;H\right)  \right)  $ to $M^{2}.$ Similar arguments allow
to prove that $M^{3,n}$ converges in the weak topology on $L^{2}\left(
\left[  t,T\right]  ;L^{2}\left(  \Omega;H\right)  \right)  $ to $M^{3}.$ The
conditions (c) and (d) in the Definition 3 of approximate mild solutions imply
that $M^{4,n}$ converges strongly to $0$ in $L^{2}\left(  \left[  t,T\right]
;L^{2}\left(  \Omega;H\right)  \right)  .$ It follows that the limit $\left(
Y,u\right)  $ satisfies
\begin{align}
Y(s) &  =S\left(  s-t\right)  \xi+\int_{t}^{s}S\left(  s-r\right)  Bu\left(
r\right)  dr+\int_{t}^{s}S\left(  s-r\right)  CY\left(  r\right)
dW_{r}\nonumber\\
&  +\int_{t}^{s}S\left(  s-r\right)  Du\left(  r\right)  dW_{r}.\label{mild1}%
\end{align}
This equation is, a priori, satisfied $dPds-$almost everywhere. We can now
identify $Y$ with its continuous version $X^{t,\xi,u}$.

\underline{Step 2}. We claim that $Y_{s}\in K$ $dPds-$ almost everywhere on
$\Omega\times\left[  t,T\right]  .$ Indeed, let us consider the following
application $\gamma:L^{2}\left(  \left[  t,T\right]  ;L^{2}\left(
\Omega;H\right)  \right)  \longrightarrow%
\mathbb{R}
_{+},$%
\[
\gamma\left(  Z\right)  =E\left[  \int_{t}^{T}d_{K}^{2}\left(  Z\left(
s\right)  \right)  ds\right]  .
\]
Obviously, this application is convex$.$ Using the fact that $Y_{n}$ converges
in the weak topology on $L^{2}\left(  \left[  t,T\right]  ;L^{2}\left(
\Omega;H\right)  \right)  $ to $Y$, one finds a sequence of convex
combinations of $\left(  Y_{n}\right)  ,$ denoted by $Z_{n},$ which converges
strongly to $Y$ in $L^{2}\left(  \left[  t,T\right]  ;L^{2}\left(
\Omega;H\right)  \right)  .$ For every $n\geq2,$%
\[
\gamma\left(  Y_{n}\right)  \leq E\left[  \int_{t}^{T}\left\vert Y_{n}\left(
s\right)  -Y_{n}\left(  \sigma\left(  s\right)  \right)  \right\vert
^{2}ds\right]  \leq Tn^{-1}.
\]
Thus, using the convexity of $\gamma$, one can assume, without loss of
generality, that
\begin{equation}
\gamma\left(  Z_{n}\right)  \leq n^{-1}, \label{gammaZn}%
\end{equation}
for all $n\geq2$. Then%
\[
\gamma\left(  Y\right)  \leq2\left(  \gamma\left(  Z_{n}\right)  +E\left[
\int_{t}^{T}\left\vert Y\left(  s\right)  -Z_{n}\left(  s\right)  \right\vert
^{2}ds\right]  \right)  .
\]
We let $n\rightarrow\infty$ in the last inequality. Due to (\ref{gammaZn}) and
to the strong convergence of $Z_{n}$ to $Y$, one has
\[
\gamma\left(  Y\right)  =0.
\]
In other words, $Y\left(  s\right)  \in K,$ $dPds-$ almost everywhere on
$\Omega\times\left[  t,T\right]  .$ The conclusion follows from the continuity
of $Y$.
\end{proof}

\subsection{Nagumo's stochastic theorem. Viability of smooth sets}

We assume that $A$ is a linear operator on $H$ that generates a semigroup of
continuous operators denoted by $\left(  S\left(  t\right)  \right)  _{t\geq
0}$. We consider $F:H\longrightarrow H$ and $G:H\longrightarrow L_{2}\left(
\Xi;H\right)  $ such that, for some positive constant $c>0,$%
\begin{equation}
\left\vert F\left(  x\right)  -F\left(  y\right)  \right\vert +\left\vert
G(x)-G(y)\right\vert _{L_{2}\left(  \Xi;H\right)  }\leq c\left\vert
x-y\right\vert . \tag{B}\label{B1}%
\end{equation}
for all $x,y\in H.$ We consider the stochastic semilinear equation%
\begin{equation}
\left\{
\begin{array}
[c]{c}%
dX^{t,\xi}(s)=\left(  AX^{t,\xi}(s)+F\left(  X^{t,\xi}(s)\right)  \right)
ds\\
+G\left(  X^{t,\xi}(s)\right)  dW_{s},\text{ for all }s\in\left[  t,T\right]
,\\
X^{t,\xi}(t)=\xi\in L^{2}\left(  \Omega,\mathcal{F}_{t},P;H\right)  .
\end{array}
\right.  \label{SDE3}%
\end{equation}
The aim of this subsection is to deduce explicit conditions for the viability
of smooth sets with respect to the system (\ref{SDE3}). In particular, we
study the viability of the closed unit ball $\overline{B}\left(  0,1\right)
.$

We begin by stating the following version of Nagumo's stochastic theorem.

\begin{theorem}
(Nagumo's stochastic theorem) A nonempty, closed set $K\subset H$ is viable
with respect to (\ref{SDE3}) if and only if, for every $t\in\left[
0,T\right)  $ and every $\xi\in L^{2}\left(  \Omega,\mathcal{F}_{t}%
,P;K\right)  ,$ there exist a sequence $h_{n}\searrow0$ and a sequence of
random variables $p_{n}\in L^{2}\left(  \Omega,\mathcal{F}_{t+h_{n}%
},P;H\right)  $ such that
\[
\lim_{n}E\left[  \left\vert p_{n}\right\vert ^{2}\right]  +\frac{1}{h_{n}%
}E\left[  \left\vert E^{\mathcal{F}_{t}}\left[  p_{n}\right]  \right\vert
^{2}\right]  =0
\]
and $S\left(  h_{n}\right)  \xi+h_{n}F\left(  \xi\right)  $ $+G\left(
\xi\right)  \left(  W_{t+h_{n}}-W_{t}\right)  +\sqrt{h_{n}}p_{n}\in K,$
$dP-$almost surely, for all $n.$
\end{theorem}

This result is a simple consequence of Theorems 1 and 2.

We consider some orthonormal basis in $\Xi,$ denoted by $\left(  e_{m}\right)
_{m\in%
\mathbb{N}
^{\ast}}$. We introduce, for every $m\in%
\mathbb{N}
^{\ast},$ $\beta_{\cdot}^{m}=\left\langle W_{\cdot},e_{m}\right\rangle .$ This
defines a sequence of independent standard $1$-dimensional Brownian motions.
We also consider, for every $m\in%
\mathbb{N}
^{\ast}$, the $m-$dimensional Brownian motion $W_{\cdot}^{m}=%
{\textstyle\sum\limits_{i=1}^{m}}
\beta_{\cdot}^{i}e_{i}$. For every $m\geq1$, we denote by $\left\{
\mathcal{F}_{t}^{m}:t\geq0\right\}  $ the filtration generated by the
one-dimensional Brownian motion $\beta_{\cdot}^{m}$ completed by the $P$-null
sets and by $\left\{  \mathcal{F}_{t}^{1,m}:t\geq0\right\}  $ the filtration
generated by the finite-dimensional Brownian motion $W_{\cdot}^{m}$ completed
by the $P$-null sets.

We will show how deterministic criteria can be inferred from the
quasi-tangency condition for smooth sets of constraints. To be more precise,
we consider sets of constraints $K$ that can be written as%
\begin{equation}
K=\left\{  x\in H:\varphi\left(  x\right)  \leq0\right\}  , \label{Kphi}%
\end{equation}
for some $C^{2}$-function $\varphi:H\longrightarrow%
\mathbb{R}
$ with bounded, Lipschitz-continuous second order derivative.

\begin{proposition}
Under the assumption (B), if the set $K$ given by (\ref{Kphi}) is viable with
respect to (\ref{SDE3}), then, for every $x\in D\left(  A\right)  $ such that
$\varphi\left(  x\right)  =0,$%
\begin{align}
0  &  \geq\left\langle D\varphi\left(  x\right)  ,Ax\right\rangle
+\left\langle D\varphi\left(  x\right)  ,F\left(  x\right)  \right\rangle
\nonumber\\
&  +\frac{1}{2}\sum_{i=1}^{\infty}\left\langle D^{2}\varphi\left(  x\right)
G\left(  x\right)  e_{i},G\left(  x\right)  e_{i}\right\rangle , \label{dn1}%
\end{align}
and
\begin{equation}
G^{\ast}\left(  x\right)  D\varphi\left(  x\right)  =0. \label{dn2}%
\end{equation}

\end{proposition}

\begin{proof}
If the set $K$ is viable, and $x\in D\left(  A\right)  $ such that
$\varphi\left(  x\right)  =0,$ for every sequence $h_{n}\searrow0$ we get, by
Nagumo's stochastic theorem (see also Remark 3), the existence of some
sequence $p_{n}\in L^{2}\left(  \Omega,\mathcal{F}_{h_{n}},P;H\right)  $ such
that
\[
\left\{
\begin{array}
[c]{l}%
\lim_{n}\left(  E\left[  \left\vert p_{n}\right\vert ^{2}\right]  +\frac
{1}{h_{n}}E\left[  \left\vert E^{\mathcal{F}_{t}}\left[  p_{n}\right]
\right\vert ^{2}\right]  \right)  =0\text{ and }\\
S\left(  h_{n}\right)  x+h_{n}F\left(  x\right)  +G\left(  x\right)  W_{h_{n}%
}+\sqrt{h_{n}}p_{n}\in K,\text{ }dP\text{-a.s.,}%
\end{array}
\right.
\]
for every $n\geq1.$ Moreover, we can assume (cf. Remark 3) that $p_{n}\in
L^{4}\left(  \Omega,\mathcal{F}_{h_{n}},P;H\right)  $, for every $n\geq1.$
Using the martingale representation theorem, there exists an $\left(
\mathcal{F}\right)  $-predictable process $q_{n}$ such that $p_{n}=E\left[
p_{n}\right]  +\int_{0}^{h_{n}}q_{n}\left(  s\right)  dW_{s}.$ Moreover, by
standard estimates,%
\begin{equation}
E\left[  \left(  \int_{0}^{h_{n}}\left\vert q_{n}\left(  s\right)  \right\vert
^{2}ds\right)  ^{2}\right]  \leq CE\left[  \left\vert p_{n}\right\vert
^{4}\right]  , \label{qn}%
\end{equation}
for some generic constant $C$ independent of $p_{n}$ and $h_{n}.$ For every
$n\in%
\mathbb{N}
^{\ast}$, we introduce the process $Y_{\cdot}^{x,n}$ given by
\[
Y_{s}^{x,n}=S\left(  h_{n}\right)  x+h_{n}F\left(  x\right)  +G\left(
x\right)  W_{s}+\sqrt{h_{n}}E\left[  p_{n}\right]  +\sqrt{h_{n}}\int_{0}%
^{s}q_{n}\left(  r\right)  dW_{r},
\]
for every $s\in\left[  0,h_{n}\right]  $. It is obvious from (\ref{qn}) that
\[
E\left[  \sup_{s\in\left[  0,h_{n}\right]  }\left\vert Y_{s}^{x,n}\right\vert
^{4}\right]  \leq C.
\]
We apply It\^{o}'s formula to $Y^{x,n}$ to get%
\begin{align}
&  \varphi\left(  Y_{h_{n}}^{x,n}\right) \nonumber\\
&  =\varphi\left(  S\left(  h_{n}\right)  x+h_{n}F\left(  x\right)
+\sqrt{h_{n}}E\left[  p_{n}\right]  \right) \nonumber\\
&  +\int_{0}^{h_{n}}\left\langle D\varphi\left(  Y_{s}^{x,n}\right)  ,\left(
G\left(  x\right)  +\sqrt{h_{n}}q_{n}\left(  s\right)  \right)  dW_{s}%
\right\rangle \nonumber\\
&  +\frac{1}{2}\int_{0}^{h_{n}}\sum_{i=1}^{\infty}\left\langle D^{2}%
\varphi\left(  Y_{s}^{x,n}\right)  \left(  G\left(  x\right)  +\sqrt{h_{n}%
}q_{n}\left(  s\right)  \right)  e_{i},\left(  G\left(  x\right)  +\sqrt
{h_{n}}q_{n}\left(  s\right)  \right)  e_{i}\right\rangle ds. \label{d0}%
\end{align}
Using a first order Taylor expansion of $\varphi$ gives
\begin{align}
&  \varphi\left(  S\left(  h_{n}\right)  x+h_{n}F\left(  x\right)
+\sqrt{h_{n}}E\left[  p_{n}\right]  \right) \nonumber\\
&  \geq\left\langle D\varphi\left(  x\right)  ,\left(  S\left(  h_{n}\right)
-I\right)  x+h_{n}F\left(  x\right)  \right\rangle -h_{n}P_{n}^{1}, \label{d1}%
\end{align}
where $P_{n}^{1}\in%
\mathbb{R}
_{+}$ such that $\lim_{n}P_{n}^{1}=0.$ Moreover, by standard estimates,
\begin{align}
&  \int_{0}^{h_{n}}\sum_{i=1}^{\infty}\left\langle D^{2}\varphi\left(
Y_{s}^{x,n}\right)  \left(  G\left(  x\right)  +\sqrt{h_{n}}q_{n}\left(
s\right)  \right)  e_{i},\left(  G\left(  x\right)  +\sqrt{h_{n}}q_{n}\left(
s\right)  \right)  e_{i}\right\rangle ds\nonumber\\
&  \geq h_{n}\sum_{i=1}^{\infty}\left\langle D^{2}\varphi\left(  x\right)
G\left(  x\right)  e_{i},G\left(  x\right)  e_{i}\right\rangle -h_{n}P_{n}%
^{2}, \label{d2}%
\end{align}
where $P_{n}^{2}\in L^{1}\left(  \Omega,\mathcal{F}_{h_{n}},P;%
\mathbb{R}
\right)  $ such that $lim_{n}E\left[  \left\vert P_{n}^{2}\right\vert \right]
=0.$ Substituting (\ref{d1}) and (\ref{d2}) in (\ref{d0}) and taking
expectation, we get
\begin{align*}
0  &  \geq E\left[  \varphi\left(  Y_{h_{n}}^{x,n}\right)  \right]
\geq\left\langle D\varphi\left(  x\right)  ,\left(  S\left(  h_{n}\right)
-I\right)  x+h_{n}F\left(  x\right)  \right\rangle \\
&  \text{ \ \ }+h_{n}\sum_{i=1}^{\infty}\left\langle D^{2}\varphi\left(
x\right)  G\left(  x\right)  e_{i},G\left(  x\right)  e_{i}\right\rangle
-h_{n}\left(  P_{n}^{1}+P_{n}^{2}\right)  .
\end{align*}
The condition (\ref{dn1}) follows by dividing the last inequality by $h_{n}$
and letting $n\rightarrow\infty.$

Standard estimates and (\ref{d0}) yield%
\begin{align}
0  &  \geq\varphi\left(  Y_{h_{n}}^{x,n}\right)  =\int_{0}^{h_{n}}\left\langle
D\varphi\left(  Y_{s}^{x,n}\right)  ,\left(  G\left(  x\right)  +\sqrt{h_{n}%
}q_{n}\left(  s\right)  \right)  dW_{s}\right\rangle \nonumber\\
&  \text{ \ \ \ \ \ \ \ \ \ \ \ \ \ \ }-\sqrt{h_{n}}Q_{n}^{1}, \label{d3'}%
\end{align}
where $Q_{n}^{1}\in L^{1}\left(  \Omega,\mathcal{F}_{h_{n}},P;%
\mathbb{R}
\right)  $ such that $lim_{n}E\left[  \left\vert Q_{n}^{1}\right\vert \right]
=0.$ Moreover, again by standard estimates,
\begin{align}
&  \int_{0}^{h_{n}}\left\langle D\varphi\left(  Y_{s}^{x,n}\right)  ,\left(
G\left(  x\right)  +\sqrt{h_{n}}q_{n}\left(  s\right)  \right)  dW_{s}%
\right\rangle \nonumber\\
&  =\left\langle D\varphi\left(  x\right)  ,G\left(  x\right)  W_{h_{n}%
}\right\rangle -\sqrt{h_{n}}Q_{n}^{2}, \label{d4}%
\end{align}
where $Q_{n}^{2}\in L^{1}\left(  \Omega,\mathcal{F}_{h_{n}},P;H\right)  $ such
that $lim_{n}E\left[  \left\vert Q_{n}^{2}\right\vert \right]  =0.$ Then, by
substituting (\ref{d4}) in (\ref{d3'}) and dividing by $\sqrt{h_{n}}$ we get%
\[
Q_{n}^{1}+Q_{n}^{2}\geq\frac{1}{\sqrt{h_{n}}}\left\langle D\varphi\left(
x\right)  ,G\left(  x\right)  W_{h_{n}}\right\rangle .
\]
For every $m\in%
\mathbb{N}
^{\ast},$ taking conditional expectation with respect to $\mathcal{F}_{\infty
}^{m}$ yields
\[
Q_{n}\geq\left\langle D\varphi\left(  x\right)  ,G\left(  x\right)
e_{m}\right\rangle \frac{\beta_{h_{n}}^{m}}{\sqrt{h_{n}}},
\]
for some process $Q_{n}\in L^{1}\left(  \Omega,\mathcal{F}_{h_{n}},P;H\right)
$ such that $lim_{n}E\left[  \left\vert Q_{n}\right\vert \right]  =0.$
Standard properties of Brownian sample paths imply that the above inequality
can only hold if (\ref{dn2}) holds true. The proof is now complete.
\end{proof}

Using the above result, we are going to deduce explicit deterministic
conditions for the viability of the unit ball. In the case of a diagonal
operator $A$, using the quasi-tangency characterization of viability, we prove
that these conditions are also sufficient.

\begin{proposition}
If the unit ball $\overline{B}\left(  0,1\right)  $ is viable with respect to
(\ref{SDE3}), then, for every $x\in D\left(  A\right)  $ such that $\left\vert
x\right\vert =1,$
\begin{equation}
\left\{
\begin{array}
[c]{l}%
0\geq\left\langle x,Ax\right\rangle +\left\langle x,F\left(  x\right)
\right\rangle +\frac{1}{2}\left\vert G\left(  x\right)  \right\vert
_{L_{2}\left(  \Xi;H\right)  }^{2},\text{ and}\\
G^{\ast}\left(  x\right)  x=0.
\end{array}
\right.  \label{n}%
\end{equation}
Conversely, if $A$ is a diagonal operator, this condition is also sufficient
for the viability of $\overline{B}\left(  0,1\right)  $.
\end{proposition}

\begin{proof}
That viability implies (\ref{n}) is a simple consequence of Proposition 4.
Conversely, we assume that (\ref{n}) holds true whenever $x\in D\left(
A\right)  $ is such that $\left\vert x\right\vert =1.$ We fix $t\in\left[
0,T\right)  $ and $\xi\in L^{2}\left(  \Omega,\mathcal{F}_{t},P;\overline
{B}\left(  0,1\right)  \right)  .$ For every $l\geq1,$ we denote by $J_{l}$
the orthogonal projector onto the space $H_{l}$ generated by the first $l$
eigenvectors of $A.$ We fix, for the moment $m,l\geq1.$ For every $t\geq0$ and
every $\zeta\in L^{2}\left(  \Omega,\mathcal{F}_{t}^{1,m},P;H_{l}\right)  $,
we define the system
\[
\left\{
\begin{array}
[c]{l}%
X_{s}^{t,\zeta,m,l}=\left(  A_{l}X_{s}^{t,\zeta,m,l}+J_{l}F\left(
X_{s}^{t,\zeta,m,l}\right)  \right)  ds+J_{l}G\left(  X_{s}^{t,\zeta
,m,l}\right)  dW_{s}^{m},\text{ }s\in\left[  t,T\right]  ,\\
X_{t}^{t,\xi,m,l}=\zeta.
\end{array}
\right.
\]
Obviously, the solution lives in the finite dimensional space $H_{l}$ and the
Brownian motion $W_{\cdot}^{m}$ is finite dimensional. Consequently, the
condition (\ref{n}) and classical results due to Aubin, Da Prato in the finite
dimensional framework (cf. \cite{ADP}, page 596) yield that the unit ball of
$H_{l},$ denoted by $\overline{B}_{H_{l}}\left(  0,1\right)  $ is viable with
respect to the previous equation. Then, from Theorem 1, for every $h>0$, there
exists $p_{l,m,h}\in L^{2}\left(  \Omega,\mathcal{F}_{t+h}^{1,m}%
,P;H_{l}\right)  \subset$ $L^{2}\left(  \Omega,\mathcal{F}_{t+h},P;H\right)  $
such that
\begin{align}
&  S\left(  h\right)  J_{l}E^{\mathcal{F}_{\infty}^{1,m}}\left[  \xi\right]
+\int_{t}^{t+h}S\left(  t+h-s\right)  J_{l}F\left(  J_{l}E^{\mathcal{F}%
_{\infty}^{1,m}}\left[  \xi\right]  \right)  ds\nonumber\\
&  +\int_{t}^{t+h}S\left(  t+h-s\right)  J_{l}G\left(  J_{l}E^{\mathcal{F}%
_{\infty}^{1,m}}\left[  \xi\right]  \right)  dW_{s}^{m}+\sqrt{h}p_{l,m,h}%
\in\overline{B}\left(  0,1\right)  , \label{r1}%
\end{align}
$dP-a.s.$ Moreover (see inequality (\ref{condp})),
\begin{align}
&  E\left[  \left\vert p_{l,m,h}\right\vert ^{2}\right]  +\frac{1}{h}E\left[
\left\vert E^{\mathcal{F}_{t}}\left[  p_{l,m,h}\right]  \right\vert
^{2}\right] \nonumber\\
&  \leq C\left(  \sup_{s\in\left[  0,h\right]  }E\left[  \left\vert \left(
S\left(  s\right)  -I\right)  J_{l}E^{\mathcal{F}_{\infty}^{1,m}}\left[
\xi\right]  \right\vert ^{2}\right]  +h\right) \nonumber\\
&  \leq C\left(  \sup_{s\in\left[  0,h\right]  }E\left[  \left\vert \left(
S\left(  s\right)  -I\right)  \xi\right\vert ^{2}\right]  +h\right)  .
\label{r2'}%
\end{align}
A careful look at the proof of Theorem 1 shows that $C$ can be chosen
independent of $l$ and $m$. The inequality (\ref{r2'}) guarantees the
existence of some subsequence $\left(  p_{l,m,h}\right)  _{l,m}$ converging
weakly in $L^{2}\left(  \Omega,\mathcal{F}_{t+h},P;H\right)  $ to some
$p_{h}\in L^{2}\left(  \Omega,\mathcal{F}_{t+h},P;H\right)  $ such that
\begin{equation}
E\left[  \left\vert p_{h}\right\vert ^{2}\right]  +\frac{1}{h}E\left[
\left\vert E^{\mathcal{F}_{t}}\left[  p_{h}\right]  \right\vert ^{2}\right]
\leq C\left(  E\left[  \sup_{s\in\left[  0,h\right]  }\left\vert S\left(
s\right)  \xi-\xi\right\vert ^{2}\right]  +h\right)  . \label{r3}%
\end{equation}
Then, by taking weak limit in $L^{2}\left(  \Omega,\mathcal{F}_{t+h}%
,P;H\right)  $ in (\ref{r1}), one gets
\begin{align}
S\left(  h\right)  \xi &  +\int_{t}^{t+h}S\left(  t+h-s\right)  F\left(
\xi\right)  ds\nonumber\\
&  +\int_{t}^{t+h}S\left(  t+h-s\right)  G\left(  \xi\right)  dW_{s}+\sqrt
{h}p_{h}\in\overline{B}\left(  0,1\right)  . \label{r4}%
\end{align}
We recall that $t,\xi,h$ are arbitrary and the constant $C$ in (\ref{r3}) can
be chosen independent of $h.$ Then (\ref{r4}) and (\ref{r3}) give the $0$
quasi-tangency condition. The viability of the initial system follows from
Theorem 2.
\end{proof}

\end{document}